\documentclass[11pt]{article}
\usepackage{xcolor}
\usepackage{mathrsfs}
\usepackage{amsmath,amsthm,amssymb,amscd}
\usepackage{booktabs} 
\usepackage{float}
\usepackage[hyperfootnotes=false, colorlinks, linkcolor={blue}, citecolor={magenta}, filecolor={blue}, urlcolor={blue}, plainpages=false, pdfpagelabels]{hyperref}
\usepackage[overload]{textcase} 
\usepackage{mathtools}
\usepackage[numbers,sort&compress]{natbib} 
\usepackage{etoolbox} 
\apptocmd{\sloppy}{\hbadness 10000\relax}{}{} 
\usepackage[left=2.5cm, right=2.5cm, top=3cm]{geometry}
\allowdisplaybreaks
\usepackage{enumerate}
\usepackage[shortlabels]{enumitem}
\usepackage{bm}
\usepackage{url}
\usepackage[multiple]{footmisc}

\usepackage{colonequals}
\usepackage{longtable} 
\usepackage[ampersand]{easylist}
\usepackage{ltablex,booktabs}
\newcommand{\A}{{\mathbb A}}
\newcommand{\Q}{{\mathbb Q}}
\newcommand{\Z}{{\mathbb Z}}
\newcommand{\R}{{\mathbb R}}
\newcommand{\C}{{\mathbb C}}

\renewcommand{\H}{\mathcal{H}}
\newcommand{\HH}{\mathbb{H}}
\newcommand{\bs}{\backslash}

\newcommand{\p}{\mathfrak p}
\newcommand{\OF}{{\mathfrak o}}
\newcommand{\GL}{{\rm GL}}
\newcommand{\PGL}{{\rm PGL}}
\newcommand{\SL}{{\rm SL}}
\newcommand{\SO}{{\rm SO}}

\newcommand{\sgn}{{\rm sgn}}
\newcommand{\St}{{\rm St}}
\newcommand{\triv}{1}

\newcommand{\M}[1]{{\rm M}(\p^{#1})}

\newcommand{\mat}[4]{{\setlength{\arraycolsep}{0.5mm}\left[\begin{smallmatrix}#1&#2\\#3&#4\end{smallmatrix}\right]}}
\newcommand{\matb}[4]{{\setlength{\arraycolsep}{0.5mm}\big[\begin{smallmatrix}#1&#2\\#3&#4\end{smallmatrix}\big]}}
\newcommand{\forget}[1]{}
\def\qdots{\mathinner{\mkern1mu\raise0pt\vbox{\kern7pt\hbox{.}}\mkern2mu
\raise3.4pt\hbox{.}\mkern2mu\raise7pt\hbox{.}\mkern1mu}}

\newtheorem{lemma}{Lemma}[section]
\newtheorem{theorem}[lemma]{Theorem}
\newtheorem{corollary}[lemma]{Corollary}
\newtheorem{proposition}[lemma]{Proposition}

\newtheorem{remark}[lemma]{Remark}

\newcommand\blfootnote[1]{%
	\begingroup
	\renewcommand\thefootnote{}\footnote{#1}%
	\addtocounter{footnote}{-1}%
	\endgroup
}
\usepackage[toc,page, title, titletoc]{appendix}
\makeatother
\makeatletter
\newcommand\appendix@section[1]{%
	\refstepcounter{section}%
	\orig@section*{Appendix \@Alph\c@section: #1}%
	\addcontentsline{toc}{section}{Appendix \@Alph\c@section: #1}%
}
\g@addto@macro\appendix{\let\section\appendix@section}
\let\orig@section\section
\makeatother
\makeatletter
\def\thickhline{%
  \noalign{\ifnum0=`}\fi\hrule \@height \thickarrayrulewidth \futurelet
   \reserved@a\@xthickhline}
\def\@xthickhline{\ifx\reserved@a\thickhline
               \vskip\doublerulesep
               \vskip-\thickarrayrulewidth
             \fi
      \ifnum0=`{\fi}}
\makeatother

\title{Classical and adelic Eisenstein series}
\author{Manami Roy, Ralf Schmidt, and Shaoyun Yi}
\date{}

\begin{document}
	
\maketitle
\blfootnote{2020 Mathematics Subject Classification: Primary 11F11, 11F37, 11F41, 11F70 \\ \hspace*{0.2in } Key words and phrases. Classical Eisenstein series, adelic Eisenstein series, automorphic representations.}
\begin{abstract}
We carry out ``Hecke summation'' for the classical Eisenstein series $E_k$ in an adelic setting. The connection between classical and adelic functions is made by explicit calculations of local and global intertwining operators and Whittaker functions. In the process we determine the automorphic representations generated by the $E_k$, in particular for $k=2$, where the representation is neither a pure tensor nor has finite length. We also consider Eisenstein series of weight $2$ with level, and Eisenstein series with character.
\end{abstract}

\tableofcontents

\section{Introduction}
Let $F\in S_k(\Gamma_0(N))$ be a cusp form of weight $k$ and level $N$, assumed to be an eigenform for almost all Hecke operators. By strong approximation, there exists a unique function $\Phi$ on $G(\A)$, where $G=\GL(2)$ and $\A$ is the ring of adeles of $\Q$, such that $\Phi$ is left invariant under $G(\Q)$, right invariant under $G(\Z_p)$ for all primes $p\nmid N$ and under the local congruence subgroups $\Gamma_0(p^{v_p(N)}\Z_p)$ for $p\mid N$, and such that
\begin{equation}\label{introeq1}
 F(z)=y^{-k/2}\Phi(\mat{1}{x}{}{1}\mat{y}{}{}{1}),\qquad z=x+iy.
\end{equation}
We say that $\Phi$ is the automorphic form corresponding to $F$. The group $G(\A)$, or more precisely the global Hecke algebra $\H$, acts on $\Phi$ by right translation, generating a representation $\pi$. This $\pi$ turns out to be irreducible, resulting in a factorization $\pi\cong\bigotimes\pi_v$ over the places of $\Q$, with representations $\pi_v$ of the local Hecke algebras $\H_v$. Since $F$ is holomorphic, the archimedean $\pi_\infty$ is the discrete series representation $\mathcal{D}_{k-1}^{\rm hol}$ with lowest weight $k$. For finite primes $p\nmid N$, the $\pi_p$ are spherical with Satake parameters related to the Hecke eigenvalues of $F$.

The standard proof that $\pi$ is irreducible goes as follows (see \cite[Thm.~5.19]{Gelbart1975}, \cite[Thm.~3.6.1]{Bump1997}). Since $\Phi$ is a cuspidal automorphic form, it lies in $L^2(G(\Q)Z(\A)\backslash G(\A))$, where $Z$ denotes the center of $G$. As a consequence, $\pi$ decomposes into a direct sum $\bigoplus\pi_i$, actually finite, with irreducible $\pi_i$. The $\pi_i$ are all near-equivalent, meaning if we factor $\pi_i\cong\bigotimes\pi_{i,v}$, then for any pair of indices $(i,j)$ we have $\pi_{i,p}\cong\pi_{j,p}$ for almost all $p$. Now one invokes the strong multiplicity one theorem to conclude that $\pi_i$ and $\pi_j$ are identical. In other words, $\pi$ must be irreducible.

Clearly, this proof does not work for non-cusp forms. For one, the corresponding automorphic form $\Phi$ may no longer be square-integrable. Also, the strong multiplicity one theorem is a result for cusp forms only (or for isobaric non-cusp forms \cite[\S~2]{Langlands1977}, but the representations involved here are not isobaric). Hence, even for the full-level holomorphic Eisenstein series
\begin{equation}\label{introeq2}
 E_k(z)=\frac1{2\zeta(k)}\sum_{\substack{c,d\in\Z\\(c,d)\neq(0,0)}}\frac1{(cz+d)^k},
\end{equation}
where $k\geq4$ is an even integer, it is not obvious that the corresponding automorphic form $\Phi_k$ generates an irreducible representation.

One strategy to prove irreducibility starts with the global parabolically induced representation $V_s:=|\cdot|^s\times|\cdot|^{-s}$, where $s$ is a complex parameter. Assuming that ${\rm Re}(s)>1/2$ to assure absolute convergence, one can construct the adelic Eisenstein series
\begin{equation}\label{introeq3}
 E(g,f)=\sum_{\gamma\in B(\Q)\backslash G(\Q)}f(\gamma g),\qquad g\in G(\A),
\end{equation}
where $B$ denotes the upper triangular subgroup of $G$. Evidently, the map $f\mapsto E(\cdot,f)$ is an intertwining operator from $V_s$ to the space of automorphic forms. Now for an appropriately chosen weight-$k$ function $f_k\in V_{(k-1)/2}$ it turns out that $\Phi_k=E(\cdot,f_k)$ is the automorphic form corresponding to $E_k$; see Theorem~\ref{Ektheorem}. Since $f_k$ generates an irreducible representation, which is easily identified, the intertwining property implies that $\Phi_k$ generates the same representation. In this manner one has proved irreducibility for $k\geq4$. See Corollary~\ref{Ektheoremcor} for the precise identification of the global representation generated by $\Phi_k$.

Certainly, this approach via adelic Eisenstein series is well known. A less familiar situation occurs for weight $k=2$, and indeed it is this case which provided the original motivation for this work. Recall that $E_2$ is a non-holomorphic modular form of weight $2$, given by the conditionally convergent series
\begin{equation}\label{introeq4}
 E_2(z)=-\frac3{\pi y}+\frac1{2\zeta(2)}\,\sum_{c\in\Z}\sum_{\substack{d\in\Z\\(c,d)\neq(0,0)}}\frac1{(cz+d)^2}.
\end{equation}
(See \cite[Sect.~1.2]{DiamondShurman2005}.) What representation is generated by the corresponding automorphic form~$\Phi_2$~? Imitating the above approach, we would start with an appropriate weight-$2$ vector $f_2\in V_{1/2}$. The first difficulty we run into is that the Eisenstein series \eqref{introeq3} is no longer absolutely convergent for $s=1/2$. This difficulty can be overcome by the familiar process of analytic continuation (also known as ``Hecke summation'', pioneered in the work \cite{Hecke1927}) : One embeds $f_2$ into a ``flat section'', considers the summation \eqref{introeq3} in the region of absolute convergence, writes down the Fourier expansion of the Eisenstein series, and observes that each piece admits analytic continuation to a meromorphic function on all of $\C$. A subtlety here is that, for some $f\in V_{1/2}$, the Eisenstein series has a pole at $s=1/2$. However, for $f=f_2$ there is no pole, so that $\Phi_2:=E(\cdot,f_2)$ is well-defined via analytic continuation.

The second difficulty we encounter is whether the map $f\mapsto E(\cdot,f)$ is still $\H$-intertwining. First we have to clarify what this means, since, as mentioned, the map is not defined on all of $V_{1/2}$. In Proposition~\ref{V12pprop} we will identify a $1$-codimensional subspace $V_{1/2}'$ for which $E(\cdot,f)$ can be defined. It turns out that the map $f\mapsto E(\cdot,f)$ is not $\H$-intertwining when restricted to $V_{1/2}'$. We therefore identify an even smaller subspace $V_{1/2}''$ by excluding all weight-$0$ functions. The map $f\mapsto E(\cdot,f)$, restricted to $V_{1/2}''$, is still not quite $\H$-intertwining, but almost; see Lemma~\ref{intpropprop}.

The third difficulty in imitating the proof of the $k\geq4$ case is that the space $V_{1/2}''$ is not an irreducible $\H$-module, but in fact highly reducible. Therefore the injectivity of the map $f\mapsto E(\cdot,f)$ restricted to $V_{1/2}''$ has to be proven in a different way. Our main argument here is contained in Lemma~\ref{HfinPhilemma}. Eventually we arrive at the following result on the structure of $\H\Phi_2$.

\begin{theorem}[\text{Theorem~\ref{E2reptheorem}}] 
The global representation $\H\Phi_2$ admits a filtration $	0\subset\C\subset\H\Phi_2$, where $\C$ is the space of constant automorphic forms, and
	\begin{equation*}
		\H\Phi_2/\C\cong\mathcal{D}_1^{\rm hol}\,\otimes\,\bigotimes_{p<\infty}V_{1/2,p}
	\end{equation*}
	as $\H$-modules. Here, $\mathcal{D}_1^{\rm hol}$ is the lowest discrete series representation (lowest weight $2$) of $\PGL(2,\R)$.
\end{theorem}
Observe that the quotient $\H\Phi_2/\C$ factors into local representations analogous to the cases $k\geq4$ (one difference however being that the local representations at finite places are all reducible).
This result has been independently obtained by Horinaga, who took the broader point of view of nearly holomorphic modular forms; see \cite[Thm.~3.8]{Horinaga2021}. We stress that our methods remain elementary and do not require the general theory of Eisenstein series or the theory of nearly holomorphic modular forms.

We note that the Maass lowering operator $L$ defined in Sect.~\ref{diffopsec} annihilates $f_2$, but not its image $\Phi_2=E(\cdot,f_2)$. In fact, it sends $\Phi_2$ to a non-zero constant automorphic form. Hence the presence of the invariant subspace $\C$ is a reflection of the fact that the map $f\mapsto E(\cdot,f)$ fails to be $\H$-intertwining at the archimedean place.

It is tempting to eliminate the $\frac1y$-term in \eqref{introeq4} by forming the function $\tilde E_{2,N}(z)=E_2(z)-NE_2(Nz)$ for a positive integer $N>1$. Then indeed $\tilde E_{2,N}\in M_2(\Gamma_0(N))$. In Sect.~\ref{E2Nsec} we consider the adelic origin of these modular forms with level. Since $\tilde E_{2,MN}(z)=M\tilde E_{2,N}(Mz)+\tilde E_{2,M}(z)$, it suffices to consider square-free $N$. In this case the functions $f_{2,N}\in V_{1/2}$ defined in \eqref{Eisleveleq6} are the natural candidates for an ``$f_2$ with level''. It turns out that the adelic Eisenstein series $\Phi_{2,N}:=E(\cdot,f_{2,N})$ does not correspond to $\tilde E_{2,N}$, but to a different modular form $E_{2,N}\in M_2(\Gamma_0(N))$, which we identify in Theorem~\ref{tildeENtheorem}. The $E_{2,N}$ have a somewhat more natural Fourier expansion than the $\tilde E_{2,N}$. In Proposition~\ref{tildeENrelationprop} we clarify the relationship between these two types of functions. The following result identifies the global representation $\H\Phi_{2,N}$ as a (highly reducible) tensor product of local representations. 
\begin{theorem}[\text{Theorem~\ref{E2Nreptheorem}}]
	For a square-free integer $N>1$,
	\begin{equation*}
		\H\Phi_{2,N}\cong\mathcal{D}_1^{\rm hol}\otimes\Big(\bigotimes_{p\mid N}\mathcal{D}_p\Big)\otimes\Big(\bigotimes_{p\nmid N}V_{1/2,p}\Big),
	\end{equation*}
where $\mathcal{D}_p$ is the Steinberg representation of $\GL(2,\Q_p)$.
\end{theorem}
We note that the global representations generated by the automorphic forms corresponding to the $\tilde E_{2,N}$ are in general not tensor products. 

In the final section we repeat parts of the previous theory in a modified setting involving a primitive Dirichlet character $\xi$ of conductor $u>1$ and the corresponding character $\chi=\otimes\chi_v$ of $\Q^\times\backslash\A^\times$. The relevant global representations are now $V_{s,\chi}:=\chi|\cdot|^s\times\chi^{-1}|\cdot|^{-s}$. The point is that for $s=(k-1)/2$ the archimedean component $V_{s,\chi_\infty}$ still contains the holomorphic discrete series representation $\mathcal{D}_{k-1}^{\rm hol}$ as a submodule, allowing us to construct from $V_{(k-1)/2,\chi}$ holomorphic modular forms of weight $k$ by choosing appropriate vectors $f_{k,\chi}\in V_{(k-1)/2,\chi}$ and forming an Eisenstein series. This way we obtain certain elements of $M_k(\Gamma_0(u^2))$, which are familiar from the classical theory; see Sect.~\ref{Eischarsec}. For $k=2$ we go one step further and consider natural vectors $f_{2,N,\chi}\in V_{1/2,\chi}$, where $N$ is an appropriately chosen squarefree integer. These lead to Eisenstein series in $M_2(\Gamma_0(u^2N))$ whose Fourier expansion is identified in Theorem~\ref{tildeENtheorem-char}. To the best of our knowledge these Eisenstein series have not been previously considered, but in Proposition~\ref{tildeENrelationprop char} we relate them to certain oldforms that can be found in the literature. Finally we identify the global representations generated by the Eisenstein series with character. This is now easier since the global intertwining operator is zero, implying that the map $f\mapsto E(\cdot,f)$ commutes with the action of the global Hecke algebra. 

The Hecke summation can be carried out in a similar manner for Eisenstein series on higher rank symplectic groups. While working on the degree two case we realized that a detailed analysis of the degree one case would be helpful, which provided the motivation for this work.

The structure of this paper is as follows. In Sect.~\ref{Sect:Preparations} we review some of our notation and basic theory used in various parts of the paper. In Sect.~\ref{Sect:intertwining}, we compute the local and global intertwining operators for vectors in $V_s$. In Sect.~\ref{Sect:Whittaker}, we compute the local and global Whittaker integrals for vectors in $V_s$. The calculations from Sects.~\ref{Sect:intertwining} and~\ref{Sect:Whittaker} are essential components in proving our main results on adelic and classical Eisenstein series without character in Sect.~\ref{Sect:Eisenstein series}. Finally, in Sect.~\ref{Sect:ES with character}, we treat Eisenstein series with character. 

{\bf Acknowledgements:} We would like to thank Charles Conley, Paul Garrett and Cris Poor for providing helpful comments. We would also like to thank the referee for the detailed comments and suggestions. Shaoyun Yi is supported by the National Natural Science Foundation of China (No.~12301016) and the Fundamental Research Funds for the Central Universities (No.~20720230025).
\section{Preparations}\label{Sect:Preparations}
In this section we review some basic facts about differential operators, Hecke algebras and induced representations which will be used throughout the paper.
\subsection{Differential operators}\label{diffopsec}
The Lie algebra $\mathfrak{sl}(2,\R)$ is spanned by
\begin{equation}\label{diffopeq1}
 \hat H=\mat{1}{0}{0}{-1},\qquad\hat R=\mat{0}{1}{0}{0},\qquad\hat L=\mat{0}{0}{1}{0},
\end{equation}
with the commutation relations $[\hat H,\hat R]=2\hat R$, $[\hat H,\hat L]=-2\hat L$ and $[\hat R,\hat L]=\hat H$. Its complexification $\mathfrak{sl}(2,\C)$ is spanned by
\begin{equation}\label{diffopeq2}
 H=-i\mat{0}{1}{-1}{0},\qquad R=\frac12\mat{1}{i}{i}{-1},\qquad L=\frac12\mat{1}{-i}{-i}{-1}
\end{equation}
with the commutation relations $[H,R]=2R$, $[H,L]=-2L$ and $[R,L]=H$. For $\theta\in\R$, let
\begin{equation}\label{rthetadefeq}
 r(\theta)=\mat{\cos(\theta)}{\sin(\theta)}{-\sin(\theta)}{\cos(\theta)}.
\end{equation}
For an integer $k$, let $W(k)$ be the space of smooth functions $\Phi$ on $\SL(2,\R)$ with the property
\begin{equation}\label{Wkdefeq}
 \Phi(gr(\theta))=e^{ik\theta}\,\Phi(g),\qquad\theta\in\R,\:g\in\SL(2,\R).
\end{equation}
This condition is equivalent to $H\Phi=k\Phi$. It follows that $R$ induces a map $W(k)\to W(k+2)$ and $L$ induces a map $W(k)\to W(k-2)$. Let $W$ be the space of smooth functions on the complex upper half plane $\HH$. For $\Phi\in W(k)$ we define an element $\tilde\Phi\in W$ by
\begin{equation}\label{PhitildePhieq}
 \tilde\Phi(x+iy)=y^{-k/2}\,\Phi(\mat{1}{x}{}{1}\mat{y^{1/2}}{}{}{y^{-1/2}}).
\end{equation}
The map $\Phi\mapsto\tilde\Phi$ establishes an isomorphism $W(k)\cong W$.
\begin{proposition}\label{diffopHproposition}
 Define operators $R_k,L_k$ on the space $W$ of smooth functions on $\HH$ by
 $$
  R_k=\frac ky+2i\frac{\partial}{\partial\tau},\qquad L_k=-2iy^2\frac{\partial}{\partial\bar\tau}.
 $$
 Then the diagrams
 $$
  \begin{CD}
   W(k)@>\sim>>W&\qquad\qquad&&W(k)@>\sim>>W\\
   @V{L}VV @VV{L_k}V&@V{R}VV @VV{R_k}V\\
   W(k-2)@>\sim>>W&&&W(k+2)@>\sim>>W
  \end{CD}
 $$
 are commutative.
\end{proposition}
\begin{proof}
Standard calculations.
\end{proof}
\subsection{Hecke algebras}\label{heckesec}
To have a convenient notation, we work with the local and global Hecke algebras. In the global case, we will have no opportunity in this note for base fields other than $\Q$, so we will make the definitions for this case only. Note that the Eisenstein series of weight two is holomorphic when the base field is totally real but not $\Q$; see \cite[Final Remark on p.~65]{Freitag1990}. The symbol $\A$ will always denote the ring of adeles of $\Q$.

For each prime $p$, let $\H_p$ be the local Hecke algebra at $p$, consisting of compactly supported, locally constant functions on $G(\Q_p)$. Note that these algebras are non-unital. The category of smooth $G(\Q_p)$-representations is equivalent to the category of nondegenerate (in the sense of \cite{Cartier1979}) $\H_p$-modules. We understand all $\H_p$-modules to be non-degenerate without mentioning it. We let $\H_{\rm fin}=\bigotimes\H_p$, the restricted tensor product taken over all prime numbers. Note that the restricted tensor product requires a choice of distinguished vector at almost every place; we always take the characteristic function of $K_p=\GL(2,\Z_p)$ to be the distinguished vector.

We will use the following notations for archimedean objects:
\begin{equation}\label{archnoteq1}
 G(\R)=\GL(2,\R),\qquad\mathfrak{g}=\mathfrak{gl}(2,\R),\qquad K_\infty={\rm O}(2).
\end{equation}
There is a general notion of archimedean Hecke algebra $\H_\infty$, introduced in \cite{KnappVogan1995}, such that the category of $(\mathfrak{g},K_\infty)$-modules is isomorphic to the category of $\H_\infty$-modules. Like in the $p$-adic case $\H_\infty$ is non-unital. However, in our situation we can get by with the simpler version of $\H_\infty$ given in \cite[Defn.~4.1]{Gelbart1975}. The point is that if a vector $v$ in a $(\mathfrak{g},K_\infty)$-module has a weight already, then $\SO(2)$ acting on $v$ stays within the same one-dimensional space. We really only need to act with the universal enveloping algebra $\mathcal{U}(\mathfrak{g})$ and the group element $\mat{1}{}{}{-1}$. Hence we introduce a formal element $\varepsilon_-$ and define $\H_\infty=\mathcal{U}(\mathfrak{g})\oplus\varepsilon_-\mathcal{U}(\mathfrak{g})$, with the multiplication determined by $\varepsilon_-^2=1$ and $\varepsilon_-X\varepsilon_-=\mat{1}{}{}{-1}X\mat{1}{}{}{-1}$ for $X\in\mathfrak{g}$. Note that this version of $\H_\infty$ actually is unital.

The global Hecke algebra is $\H=\H_\infty\otimes\H_{\rm fin}$. It acts on the space $\mathcal{A}$ of automorphic forms on $G(\A):=\GL(2,\A)$. Any irreducible subquotient is called an automorphic representation. The $G(\A)$-representation generated by an automorphic form $\Phi$ is $\H\Phi$. Here, we use the word ``$G(\A)$-representation'' as a synonym for $\H$-module, even though it is not a representation of $G(\A)$ in the strict sense of the word. 

If $(\pi,V)$ is an irreducible $\H$-module, then there exist irreducible $\H_p$-modules $(\pi_p,V_p)$ for all places $p\leq\infty$, and for almost all finite $p$ a non-zero $K_p$-fixed vector $v_p$, such that $\pi$ is the restricted tensor product of the representations $\pi_v$ with respect to these distinguished vectors. We shall simply write $\pi\cong\bigotimes\pi_p$ or $V\cong\bigotimes V_p$.
\subsection{Induced representations}\label{indrepsec}
In this paper we will work exclusively with parabolically induced representations of $\GL(2)$ over a $p$-adic field, or over $\R$, or over the adeles of $\Q$. We recall some basic facts.
\subsubsection*{The non-archimedean case}
Let $F$ be a non-archimedean local field of characteristic zero. In this context we will always use the following notations. We let $\OF$ be the ring of integers of $F$, and $\p$ the maximal ideal of $\OF$. The symbol $\varpi$ denotes a generator of $\p$. The absolute value $|\cdot|$ on $F$ is normalized such that $|\varpi|=q^{-1}$, where $q=\#\OF/\p$. Let $v$ be the normalized valuation on $F$. We normalize the Haar measure on $F$ such that the volume of $\OF$ is $1$.

Let $\eta$ be a character of $F^\times$. Using a common notation, we denote by $\eta\times\eta^{-1}$ the representation of $G=\GL(2,F)$ parabolically induced by the character $\mat{a}{b}{}{d}\mapsto\eta(a/d)$ of the upper triangular subgroup $B$. Hence, the standard model of $\eta\times\eta^{-1}$ consists of smooth functions $f:G\to\C$ with the transformation property
\begin{equation}\label{indreptrafoeq}
 f(\mat{a}{b}{}{d}g)=\eta(a/d)\,|a/d|^{1/2}f(g)
\end{equation}
for $g\in G$, $a,d\in F^\times$, and $b\in F$. The group $G$ acts on this space by right translations. It is known by \cite[Thm.~3.3]{JacquetLanglands1970} or \cite[Thm.~4.5.1]{Bump1997} that $\eta\times\eta^{-1}$ is irreducible except when $\eta^2=|\cdot|^{\pm1}$. If $\eta^2=|\cdot|$, and hence $\eta=\mu|\cdot|^{1/2}$ with a quadratic character $\mu$, then there is an exact sequence
\begin{equation}\label{padicexacteq1}
 0\longrightarrow\mu\St_{\GL(2,F)}\longrightarrow\mu|\cdot|^{1/2}\times\mu|\cdot|^{-1/2}\longrightarrow\mu\triv_{\GL(2,F)}\longrightarrow0,
\end{equation}
where $\St_{\GL(2,F)}$ (resp.~$\triv_{\GL(2,F)}$) denotes the Steinberg (resp.\ trivial) representation of $\GL(2,F)$. If $\eta=\mu|\cdot|^{-1/2}$ with a quadratic character $\mu$, then there is an exact sequence
\begin{equation}\label{padicexacteq2}
 0\longrightarrow\mu\triv_{\GL(2,F)}\longrightarrow\mu|\cdot|^{-1/2}\times\mu|\cdot|^{1/2}\longrightarrow\mu\St_{\GL(2,F)}\longrightarrow0.
\end{equation}
In this latter case, the one-dimensional subspace realizing $\mu\triv_{\GL(2,F)}$ is spanned by the function $g\mapsto\mu(\det(g))$.

To have a concise notation, we let $V_s=|\cdot|^s\times|\cdot|^{-s}$ for a complex parameter $s$. Then
\begin{equation}\label{indreptrafoeq2}
 f(\mat{a}{b}{}{d}g)=\Big|\frac ad\Big|^{s+1/2}f(g)
\end{equation}
for the functions in $V_s$. For $s=1/2$ and $s=-1/2$ we have the exact sequences
\begin{align}
 \label{padicexacteq3}&0\longrightarrow\St_{\GL(2,F)}\longrightarrow V_{1/2}\longrightarrow\triv_{\GL(2,F)}\longrightarrow0,\\
 \label{padicexacteq4}&0\longrightarrow\triv_{\GL(2,F)}\longrightarrow V_{-1/2}\longrightarrow\St_{\GL(2,F)}\longrightarrow0.
\end{align}
\subsubsection*{The archimedean case}
Now consider the field $\R$ with its usual absolute value $|\cdot|$. For a complex parameter $s$, there exists a Hilbert space representation $\hat V_s$ whose space of smooth vectors consists of smooth functions $f:\GL(2,\R)\to\C$ with the transformation property \eqref{indreptrafoeq2}; see \cite[Prop.~2.5.3]{Bump1997}. We usually work with the subspace $V_s$ of $K_\infty$-finite vectors, which is a $(\mathfrak{g},K_\infty)$-module, or equivalently, an $\H_\infty$-module. As a vector space, $V_s$ has a basis consisting of the weight-$k$ functions $f_s^{(k)}$ for $k\in2\Z$, where we use the Iwasawa decomposition to define
\begin{equation}\label{localrealeq2}
 f_s^{(k)}(\mat{a}{b}{}{d}r(\theta))=\Big|\frac ad\Big|^{s+1/2}e^{ik\theta},\qquad r(\theta)=\mat{\cos(\theta)}{\sin(\theta)}{-\sin(\theta)}{\cos(\theta)}.
\end{equation}
The following result uses the Lie algebra elements defined in \eqref{diffopeq2}.
\begin{lemma}\label{RHLVslemma}
 Let $f_s^{(k)}\in V_s$ be the function in \eqref{localrealeq2}. Then, for all even integers $k$,
 \begin{align}
  \label{RHLVslemmaeq2}Hf_s^{(k)}&=kf_s^{(k)},\\
  \label{RHLVslemmaeq1}Rf_s^{(k)}&=\Big(s+\frac{1+k}2\Big)f_s^{(k+2)},\\
  \label{RHLVslemmaeq3}Lf_s^{(k)}&=\Big(s+\frac{1-k}2\Big)f_s^{(k-2)},\\
  \label{RHLVslemmaeq4}\varepsilon_-f_s^{(k)}&=f_s^{(-k)}.
 \end{align}
\end{lemma}
\begin{proof}
Standard calculations.
\end{proof}

It follows that $V_s$ is irreducible unless $s\in\frac12+\Z$. If $s=\frac{k-1}2$ with a positive even integer $k$, then there is an exact sequence
\begin{equation}\label{realexacteq1}
 0\longrightarrow\mathcal{D}_{k-1}^{\rm hol}\longrightarrow V_s\longrightarrow\mathcal{F}_{k-1}\longrightarrow0,
\end{equation}
where $\mathcal{D}_{k-1}^{\rm hol}$ is the discrete series representation of $\PGL(2,\R)$ with weight structure $[\ldots,-k-2,-k,k,k+2,\ldots]$, and $\mathcal{F}_{k-1}$ is the $(k-1)$-dimensional irreducible representation of $\PGL(2,\R)$ with weight structure $[-k+2,-k+4,\ldots,k-4,k-2]$. If $s=\frac{1-k}2$ with a positive even integer $k$, then there is an exact sequence
\begin{equation}\label{realexacteq2}
 0\longrightarrow\mathcal{F}_{k-1}\longrightarrow V_s\longrightarrow\mathcal{D}_{k-1}^{\rm hol}\longrightarrow0.
\end{equation}
We may also consider the twist of $V_s$ by the sign character of $\R^\times$, i.e, $\sgn|\cdot|^s\times\sgn|\cdot|^{-s}$. Then we have similar reducibilities and exact sequences as in \eqref{realexacteq1}, except $\mathcal{F}_{k-1}$ is replaced by the twist $\sgn\,\mathcal{F}_{k-1}$. (The discrete series representations are invariant under twisting by $\sgn$.)

For $s=1/2$ and $s=-1/2$ we have, in analogy with \eqref{padicexacteq3} and \eqref{padicexacteq4}, the exact sequences
\begin{align}
 \label{realexacteq3}&0\longrightarrow\mathcal{D}_1^{\rm hol}\longrightarrow V_{1/2}\longrightarrow\mathcal{F}_1\longrightarrow0,\\
 \label{realexacteq4}&0\longrightarrow\mathcal{F}_1\longrightarrow V_{-1/2}\longrightarrow\mathcal{D}_1^{\rm hol}\longrightarrow0,
\end{align}
where $\mathcal{D}_1^{\rm hol}$ is the lowest discrete series representation (lowest weight $2$) and $\mathcal{F}_1=\triv_{\GL(2,\R)}$ is the trivial representation.
\subsubsection*{The global case}
Let $\A$ be the ring of adeles of $\Q$. In the global context we denote the absolute value on $\Q_p$ by $|\cdot|_p$, the absolute value on $\R$ by $|\cdot|_\infty$, and let $|\cdot|=\prod_{v\leq\infty}|\cdot|_v$ be the global absolute value (the product being over the places of $\Q$). For a complex parameter $s$ and a place $v$ we have the local $\H_v$-modules $V_{s,v}$ defined above. We also have an analogous global $\H$-module $V_s$, consisting of smooth $K_\infty$-finite functions $f$ on $\GL(2,\A)$ with the transformation property \eqref{indreptrafoeq}. There is a natural isomorphism of $\H$-modules $V_s\cong\bigotimes V_{s,v}$, where we mean the restricted tensor product over the places of $\Q$. We take the unique $K_p$-invariant function $f_{s,p}^{\rm sph}\in V_{s,p}$ with the property $f_{s,p}^{\rm sph}(1)=1$ as the distinguished vector to form the restricted tensor product.

The global $V_{1/2}$ is highly reducible, since every $V_{1/2,v}$ is reducible. To have a uniform notation, we let $\mathcal{D}_v$ be the infinite-dimensional invariant subspace of $V_{1/2,v}$, and $\mathcal{F}_v=V_{1/2,v}/\mathcal{D}_v$. Hence
\begin{align}
 \label{DpFpeq1}\mathcal{D}_v&\cong\begin{cases}
                \mathcal{D}_1^{\rm hol}&\text{if }v=\infty,\\
                \St_{\GL(2,\Q_p)}&\text{if }v=p<\infty,
               \end{cases}\\
 \label{DpFpeq2}\mathcal{F}_v&\cong\begin{cases}
                \triv_{\GL(2,\R)}&\text{if }v=\infty,\\
                \triv_{\GL(2,\Q_p)}&\text{if }v=p<\infty.
               \end{cases}
\end{align}
By \cite[Lemma~1]{Langlands1979}, the irreducible subquotients of $V_{1/2}$ are $\big(\bigotimes_{v\in S}\mathcal{D}_v\big)\otimes\big(\bigotimes_{v\notin S}\mathcal{F}_v\big)$, where $S$ is a finite set of places.
\subsubsection*{Flat sections}
In the $p$-adic case let $G=\GL(2,F)$, $K=\GL(2,\OF)$, in the real case let $G=\GL(2,\R)$, $K=K_\infty$, and in the global case let $G=\GL(2,\A)$, $K=K_\infty\prod_{p<\infty}\GL(2,\Z_p)$. In either context, a family of functions $f_s\in V_s$, where $s$ runs through a complex domain $D$, is said to be a \emph{flat section} if the restriction of $f_s$ to $K$ is independent of $s$. Using the Iwasawa decomposition, we define a function $\delta:G\to\C$ by
\begin{equation}\label{deltadefeq}
 \delta(g)=\Big|\frac ad\Big|,\qquad\text{where }g=\mat{a}{b}{0}{d}\kappa,\;\kappa\in K.
\end{equation}
If $f\in V_{s_0}$, then the function $f_s:=\delta^{s-s_0}f$ lies in $V_s$, for any $s,s_0\in\C$. The family $\{f_s\}$ is then the unique flat section containing $f$.
\section{Local and global intertwining operators} \label{Sect:intertwining}
In this section we review the standard intertwining operators $A_s:V_s\to V_{-s}$ in the non-archimedean, archimedean and global case. The global intertwining operator has a simple pole at the critical point $s=1/2$. We will show in Sect.~\ref{V12psec} that it can still be evaluated on a large enough invariant subspace. The fact that it does not retain the full intertwining property is responsible for the non-holomorphy of the classical modular form $E_2$.
\subsection{Non-archimedean case}\label{interpadic}
Let $F$ be a non-archimedean local field of characteristic zero. The symbols $\OF$, $\p$, $\varpi$, $q$, $|\cdot|$ and $v$ have the same meanings as in Sect.~\ref{indrepsec}. We let $G=\GL(2,F)$ and $K=\GL(2,\OF)$. For a non-negative integer $n$, let $\Gamma_0(\p^n)=K\cap\mat{\OF}{\OF}{\p^n}{\OF}$. We fix the Haar measure on $F$ for which the volume of $\OF$ is $1$. 

For a complex parameter $s$ let $V_s=|\cdot|^s\times|\cdot|^{-s}$ be as in Sect.~\ref{indrepsec}, with the reducibilities \eqref{padicexacteq3} and \eqref{padicexacteq4}. For $f_s\in V_s$ with $q^{2s}\neq1$, we define
\begin{equation}\label{localpadiceq1}
 (A_sf_s)(g):=\lim_{N\to\infty}\bigg(\int\limits_{\substack{F\\v(b)>-N}}f_s(\mat{}{-1}{1}{}\mat{1}{b}{}{1}g)\,db+(1-q^{-1})\frac{q^{-2Ns}}{1-q^{-2s}}f_s(g)\bigg).
\end{equation}
Using the identity $\mat{}{-1}{1}{}\mat{1}{b}{}{1}=\matb{b^{-1}}{-1}{}{b}\mat{1}{}{b^{-1}}{1}$, it is easily verified that the expression in parentheses stabilizes for large enough $N$, so that the definition makes sense. Assuming that ${\rm Re}(s)>0$, a standard calculation shows that
\begin{equation}\label{localpadiceq2}
 (A_sf_s)\left(g\right)=\int\limits_Ff_s(\begin{bsmallmatrix}&-1\\1\end{bsmallmatrix}\begin{bsmallmatrix}1&b\\&1\end{bsmallmatrix}g)\,db.
\end{equation}
It is straightforward to verify from \eqref{localpadiceq2} that $A_sf_s\in V_{-s}$, so that we obtain an intertwining operator $A_s:V_s\to V_{-s}$ for ${\rm Re}(s)>0$. In fact, the intertwining property can also be verified from \eqref{localpadiceq1}, so that it holds for any $s\in\C$ such that $q^{2s}\neq1$.

Now assume that $f_s$ varies in a flat section. Then it follows from \eqref{localpadiceq1} that $(A_sf_s)(g)$ is a meromorphic function of $s$, for any fixed $g$, with possible poles at the points where $q^{2s}=1$.

\begin{remark}
 A different proof of the intertwining property
\begin{equation}\label{localpadiceq3}
 (A_sf^h)(g)=(A_sf)(gh),
\end{equation}
where $f\in V_s$ and $f^h(g)=f(gh)$, goes as follows. First, it holds for ${\rm Re}(s)>0$ by \eqref{localpadiceq2}. For other values of $s$, let $f_s$ be the flat section containing $f$. Then $(A_s(f_s)^h)(g)=(A_sf_s)(gh)$ holds for ${\rm Re}(s)>0$, and one argues that both sides are meromorphic functions of $s$. However, it is not entirely obvious that the left hand side is a meromorphic function of $s$, since in general $(f_s)^h\neq (f^h)_s$.
\end{remark}

For a compact-open subgroup $\Gamma$ of $G$, let $V_s^\Gamma$ be the finite-dimensional subspace of $V_s$ consisting of $\Gamma$-invariant vectors. The intertwining operator $A_s$ induces a map $V_s^\Gamma\to V_{-s}^\Gamma$. We consider in particular $\Gamma=\Gamma_0(\p)$. Since $G=B\Gamma_0(\p)\sqcup B\mat{}{-1}{1}{}\Gamma_0(\p)$, where $B$ is the upper triangular subgroup, the space $V_s^{\Gamma_0(\p)}$ is $2$-dimensional. We define two distinguished vectors in this space. The first is the normalized spherical vector $f_s^{\rm sph}$, characterized by $f_s^{\rm sph}(1)=1$ and being $K$-invariant. The second is the \emph{Steinberg vector}
\begin{equation}\label{localpadiceq4}
 f_s^{\rm St}=\frac1{1-q^{2s+1}}\Big((1+q^{2s})f_s^{\rm sph}-q^{s-1/2}(q+1)\mat{1}{}{}{\varpi}f_s^{\rm sph}\Big),
\end{equation}
which satisfies $f_s^{\rm St}(1)=1$ and $f_s^{\rm St}(\mat{}{-1}{1}{})=-q^{-1}$. The two vectors $f_s^{\rm sph}$ and $f_s^{\rm St}$ form a basis of $V_s^{\Gamma_0(\p)}$. Calculations show that
\begin{align}
 \label{localpadiceq5}A_sf^{\rm sph}_s&=\frac{1-q^{-2s-1}}{1-q^{-2s}}\,f^{\rm sph}_{-s},\\
 \label{localpadiceq6}A_sf_s^{\rm St}&=-q^{-1}\frac{1-q^{-2s+1}}{1-q^{-2s}}f_{-s}^{\rm St}.
\end{align}
In particular, for $s=1/2$,
\begin{equation}\label{localpadiceq7}
 f_{1/2}^{\rm St}=\frac1{1-q}\Big(f_{1/2}^{\rm sph}-\mat{1}{}{}{\varpi}f_{1/2}^{\rm sph}\Big)
\end{equation}
lies in the kernel of $A_{1/2}$. It is the newform in the Steinberg representation, explaining the name. The kernel of $A_{1/2}$ is the subrepresentation $\St_{\GL(2,F)}$ of $V_{1/2}$. The vector $f_{-1/2}^{\rm sph}$, which is a constant function, spans the kernel of $A_{-1/2}$.

\begin{remark}
 If $f\in V_{1/2}$ lies in $\St_{\GL(2,F)}$, and if $f_s$ is the flat section containing $f$, then the function $(A_sf_s)(g)$ has a zero at $s=1/2$. Therefore the definition
 \begin{equation}\label{localpadiceq21}
  (B_{1/2}f)(g):=\lim_{s\to1/2}\frac{(A_sf_s)(g)}{s-1/2}
 \end{equation}
 makes sense. It is easy to see that $B_{1/2}f\in V_{-1/2}$. It follows from \eqref{localpadiceq6} that $B_{1/2}$ is non-zero. As a consequence, $B_{1/2}$ cannot be an intertwining operator, since $V_{-1/2}$ does not contain an invariant subspace isomorphic to $\St_{\GL(2,F)}$.
\end{remark}
\subsection{Archimedean case}\label{interarch}
In the archimedean case we recall that $V_s$ is the subspace of $K_\infty$-finite vectors of the Hilbert space representation $\hat V_s=|\cdot|^s\times|\cdot|^{-s}$ of $G(\R)=\GL(2,\R)$. It is spanned by the functions $f_s^{(k)}$ defined in \eqref{localrealeq2}, for $k\in2\Z$. We consider the usual Lebesgue measure on $\mathbb{R}$. For $f\in V_s$ we define
\begin{equation}\label{localrealeq3}
 (A_sf)\left(g\right)=\int\limits_\R f(\begin{bsmallmatrix}&-1\\1\end{bsmallmatrix}\begin{bsmallmatrix}1&b\\&1\end{bsmallmatrix}g)\,db.
\end{equation}
The calculation in \cite[Prop.~2.6.2]{Bump1997} shows that the integral in \eqref{localrealeq3} is convergent for ${\rm Re}(s)>0$ (just like in the $p$-adic case) and defines a vector in $V_{-s}$.

\begin{lemma}\label{localreallemma1}
 Assume that ${\rm Re}(s)>0$ and $k\in2\Z$. Then
 \begin{equation}\label{localreallemma1eq1}
  A_sf_s^{(k)}=(-1)^{k/2}\sqrt{\pi}\,\frac{\Gamma(s)\Gamma(s+1/2)}{\Gamma(s+(1+k)/2)\Gamma(s+(1-k)/2)}f_{-s}^{(k)}.
 \end{equation}
\end{lemma}
\begin{proof}
See \cite[Prop.~2.6.3]{Bump1997}.
\end{proof}

In particular,
\begin{align}
 \label{localreallemma1eq2}A_sf_s^{(0)}&=\sqrt{\pi}\,\frac{\Gamma(s)}{\Gamma(s+1/2)}f_{-s}^{(0)},\\
 \label{localreallemma1eq3}A_sf_s^{(2)}&=-\sqrt{\pi}\,\frac{2s-1}{2s+1}\cdot\frac{\Gamma(s)}{\Gamma(s+1/2)}f_{-s}^{(2)}.
\end{align}

We can use \eqref{localreallemma1eq1} to define $A_sf_s^{(k)}$ for any $s\notin\{0,-1,-2,\ldots\}$. (The numerator has poles, some of which are canceled by poles of the denominator.) Using Lemma~\ref{RHLVslemma} one can show that the map $A_s:V_s\to V_{-s}$ thus defined is a map of $\H_\infty$-modules.

\begin{remark}
 Assume that $\hat V_s$ and $\hat V_{-s}$ are irreducible. Then the existence of the intertwining operator $A_s$ shows that $\hat V_s$ and $\hat V_{-s}$ are infinitesimally equivalent, i.e., their underlying $\H_\infty$-modules are equivalent. However, they are not isomorphic as Hilbert space representations; see \cite[Exercise~2.6.1]{Bump1997}.
\end{remark}

Let $\ell$ be an odd positive integer. It follows from \eqref{localreallemma1eq1} that
\begin{equation}\label{localrealeq6}
 A_{\ell/2}f_{\ell/2}^{(k)}=0\quad\Longleftrightarrow\quad k\in\pm\{\ell+1,\ell+3,\ldots\},
\end{equation}
showing that the kernel of $A_{\ell/2}$ is precisely the subrepresentation $\mathcal{D}_\ell^{\rm hol}$ of $V_{\ell/2}$. In particular, for $s=1/2$, the discrete series representation $\mathcal{D}_1^{\rm hol}$ is the kernel of $A_{1/2}$.

Let $V_{1/2}^{\neq0}$ be the subspace of $V_{1/2}$ spanned by the $f_s^{(k)}$ with $k\neq0$. For $f\in V_{1/2}^{\neq0}$, let $f_s$ be the unique flat section containing $f$, and define
\begin{equation}\label{intrealeq10}
 (\tilde A_{1/2}f)(g):=\lim_{s\to1/2}\zeta(2s)(A_sf_s)(g),\qquad g\in G(\R).
\end{equation}
This is well-defined, because $\lim_{s\to1/2}(A_sf_s)(g)=0$ by \eqref{localreallemma1eq1}. It is easy to see that $\tilde A_{1/2}f\in V_{-1/2}$.
\begin{lemma}\label{localreallemma2}
 Let $\tilde A_{1/2}:V_{1/2}^{\neq0}\to V_{-1/2}$ be the operator defined in \eqref{intrealeq10}. Then the following holds for all even, non-zero integers $k$.
 \begin{enumerate}
  \item
   \begin{equation}\label{localreallemma2eq1}
    \tilde A_{1/2}f_{1/2}^{(k)}=-\frac\pi{|k|}f_{-1/2}^{(k)}.
   \end{equation}
  \item
   \begin{equation}\label{localreallemma2eq2}
    \tilde A_{1/2}(Hf_{1/2}^{(k)})=H(\tilde A_{1/2}f_{1/2}^{(k)}).
   \end{equation}
  \item
   \begin{equation}\label{localreallemma2eq3}
    \tilde A_{1/2}(Rf_{1/2}^{(k)})=\begin{cases}
                                R(\tilde A_{1/2}f_{1/2}^{(k)})&\text{if }k\neq-2,\\
                                0&\text{if }k=-2.
                               \end{cases}
   \end{equation}
  \item
   \begin{equation}\label{localreallemma2eq4}
    \tilde A_{1/2}(Lf_{1/2}^{(k)})=\begin{cases}
                                L(\tilde A_{1/2}f_{1/2}^{(k)})&\text{if }k\neq2,\\
                                0&\text{if }k=2.
                               \end{cases}
   \end{equation}
  \item
   \begin{equation}\label{localreallemma2eq5}
    \tilde A_{1/2}(\varepsilon_-f_{1/2}^{(k)})=\varepsilon_-(\tilde A_{1/2}f_{1/2}^{(k)}).
   \end{equation}
 \end{enumerate}
\end{lemma}
\begin{proof}
Property \eqref{localreallemma2eq1} follows from \eqref{localreallemma1eq1}, observing that the $\Gamma$-function has residue $(-1)^n/(n!)$ at $s=-n$ for a positive integer $n$. Property \eqref{localreallemma2eq2} is immediate from \eqref{localreallemma2eq1} and \eqref{RHLVslemmaeq2}. Property \eqref{localreallemma2eq5} is immediate from \eqref{localreallemma2eq1} and \eqref{RHLVslemmaeq4}. Equation \eqref{localreallemma2eq3} holds for $k=-2$, because $Rf_{1/2}^{(-2)}=0$. For $k\geq2$ or $k\leq-4$ it follows from \eqref{localreallemma2eq1} and \eqref{RHLVslemmaeq1}. Equation \eqref{localreallemma2eq4} holds for $k=2$, because $Lf_{1/2}^{(2)}=0$. For $k\geq4$ or $k\leq-2$ it follows from \eqref{localreallemma2eq1} and \eqref{RHLVslemmaeq3}.
\end{proof}

It follows from Lemma~\ref{localreallemma2} that if we compose $\tilde A_{1/2}:V_{1/2}^{\neq0}\to V_{-1/2}$ with the projection $V_{-1/2}\to V_{-1/2}/\C\cong\mathcal{D}_1^{\rm hol}$, then we obtain an $\H_\infty$-isomorphism $V_{1/2}^{\neq0}\to\mathcal{D}_1^{\rm hol}$.
\subsection{Global case}\label{globalsec}
We now consider the global $\H$-module $V_s\cong\bigotimes V_{s,v}$.
Let the Haar measure on $\A$ be the product of the local measures chosen above. We choose the Haar measure on $\Q\bs\A$ such that the volume of $\Q\bs\A$ is $1$. We would like to define a global intertwining operator $A_s:V_s\to V_{-s}$ by the integral
\begin{equation}\label{globaleq2}
 (A_sf)\left(g\right)=\int\limits_\A f(\begin{bsmallmatrix}&-1\\1\end{bsmallmatrix}\begin{bsmallmatrix}1&b\\&1\end{bsmallmatrix}g)\,db.
\end{equation}
To investigate convergence, assume that $f$ corresponds to a pure tensor $\otimes f_v$, and that $g=(g_v)_v$. Then
\begin{equation}\label{globaleq3}
 (A_sf)\left(g\right)=\prod_v\:\int\limits_{\Q_v} f_v(\begin{bsmallmatrix}&-1\\1\end{bsmallmatrix}\begin{bsmallmatrix}1&b\\&1\end{bsmallmatrix}g_v)\,db.
\end{equation}
Let $T$ be a finite set of finite places such that $f_p=f_p^{\rm sph}$ for $p\notin T$; such a set $T$ exists by definition of the restricted tensor product. Then
\begin{align}\label{globaleq4}
 (A_sf)\left(g\right)&=\bigg(\prod_{v\in T\cup\{\infty\}}\:\int\limits_{\Q_v} f_v(\begin{bsmallmatrix}&-1\\1\end{bsmallmatrix}\begin{bsmallmatrix}1&b\\&1\end{bsmallmatrix}g_v)\,db\bigg)\bigg(\prod_{p\notin T}\:\int\limits_{\Q_p}f^{\rm sph}_{s,p}(\begin{bsmallmatrix}&-1\\1\end{bsmallmatrix}\begin{bsmallmatrix}1&b\\&1\end{bsmallmatrix}g_p)\,db\bigg)\nonumber\\
 &\stackrel{\eqref{localpadiceq5}}{=}\bigg(\prod_{v\in T\cup\{\infty\}}\:(A_{s,v}f_v)(g_v)\bigg)\bigg(\prod_{p\notin T}\frac{1-p^{-2s-1}}{1-p^{-2s}}f^{\rm sph}_{-s,p}(g_p)\bigg).
\end{align}
We see from properties of the Riemann zeta function that the infinite product converges if ${\rm Re}(s)>1/2$. Since every element of $V_s$ is a sum of pure tensors, we conclude that the integral in \eqref{globaleq2} converges provided that ${\rm Re}(s)>1/2$, and that in this region it defines an intertwining operator $A_s:V_s\to V_{-s}$.

We may rewrite \eqref{globaleq4} as
\begin{equation}\label{globaleq5}
 (A_sf)\left(g\right)=\frac{\zeta(2s)}{\zeta(2s+1)}((A_{s,\infty}f_\infty)(g_\infty))\bigg(\prod_{p\in T}\:\frac{1-p^{-2s}}{1-p^{-2s-1}}(A_{s,p}f_p)(g_p)\bigg)\bigg(\prod_{p\notin T}f^{\rm sph}_{-s,p}(g_p)\bigg).
\end{equation}
If $f=f_s=\otimes f_{s,v}$ varies in a flat section, then the only pole in the region ${\rm Re}(s)>0$ of the right hand side of \eqref{globaleq5} is at $s=1/2$, coming from the factor $\zeta(2s)$. Thus $(A_sf_s)(g)$ admits an analytic continuation to this region, with at most one simple pole at $s=1/2$. The other possible poles on $\C$ can also be determined from \eqref{globaleq5}, but we will have no need for this discussion. (Among these, the most intractable ones come from the zeros of $\zeta(2s+1)$, which is why it is tempting to normalize the intertwining operator by multiplying by this factor.)
\subsection{The spaces \texorpdfstring{$V_{1/2}'$}{} and \texorpdfstring{$V_{1/2}''$}{}}\label{V12psec}
Recall from the previous section that in ${\rm Re}(s)>0$ the global intertwining operator has only one possible pole at $s=1/2$. The next result shows that for most $f$ there is actually no pole. For $f\in V_{1/2}$, let $f_s$ be the flat section containing $f$, and define
\begin{equation}\label{V12peq1}
 (A_{1/2}f)(g):=\lim_{s\to1/2}(A_sf_s)(g),\qquad g\in G(\A),
\end{equation}
provided this limit exists. Let $V_{1/2}'$ be the subspace of $f\in V_{1/2}$ for which the limit \eqref{V12peq1} exists for all $g\in G(\A)$. In the following proof we will utilize the subspaces $U_k$ of $V_{1/2}$ defined by
\begin{equation}\label{Ukdefeq}
 U_k:=\C f_{1/2, \infty}^{(k)}\otimes\Big(\bigotimes_{p<\infty}V_{1/2,p}\Big)
\end{equation}
for $k\in2\Z$. Here $f_{1/2, \infty}^{(k)}$ is the normalized weight-$k$ function in $V_{1/2,\infty}$; see \eqref{localrealeq2}. Note that $V_{1/2}=\bigoplus_{k\in2\Z}U_k$.

\begin{proposition}\label{V12pprop}
 The space $V_{1/2}'$ is invariant under the action of $G(\A_{\rm fin})$. The space $V_{1/2}/V_{1/2}'$ is one-dimensional and carries the trivial representation of $G(\A_{\rm fin})$.
\end{proposition}
\begin{proof}
It follows from \eqref{globaleq5} and Lemma~\ref{localreallemma1} that $U_k\subset V_{1/2}'$ if $k\neq0$. For a square-free, positive integer $N$, let
\begin{equation}\label{V12ppropeq1}
 U_{0,N}=\C f_{1/2, \infty}^{(0)}\otimes\Big(\bigotimes_{p\mid N}\mathcal{D}_p\Big)\otimes\Big(\bigotimes_{p\nmid N}V_{1/2,p}\Big).
\end{equation}
where we recall $\mathcal{D}_p$ is the infinite-dimensional invariant subspace of $V_{1/2,p}$. Since $\mathcal{D}_p$ is the kernel of $A_{1/2,p}$, it follows from \eqref{globaleq5} that $U_{0,N}\subset V_{1/2}'$ if $N>1$. As vector spaces (but not as $\H_p$-modules) we have $V_{1/2,p}=\C f_{1/2,p}^{\rm sph}\oplus\mathcal{D}_p$. It follows that any element of $U_0$ can be written as a multiple of $f_{1/2}^{\rm sph}:=f_{1/2, \infty}^{(0)}\otimes\big(\otimes_{p<\infty}f_{1/2,p}^{\rm sph}\big)$ plus elements of $U_{0,N}$ for various $N>1$. In other words, if $U_0'$ is the sum of all spaces $U_{0,N}$ for all $N>1$, then $U_0=\C f_{1/2}^{\rm sph}\oplus U_0'$. (The sum is direct because $U_0'\subset V_{1/2}'$ and $f_{1/2}^{\rm sph}\notin V_{1/2}'$.) Altogether it follows that
\begin{equation}\label{V12ppropeq3}
 V_{1/2}=\C f_{1/2}^{\rm sph}\oplus U_0'\oplus\bigoplus_{\substack{k\in2\Z\\k\neq0}}U_k.
\end{equation}
It is now clear that
\begin{equation}\label{V12ppropeq4}
 V_{1/2}'=U_0'\oplus\bigoplus_{\substack{k\in2\Z\\k\neq0}}U_k.
\end{equation}
Since every $U_k$ and every $U_{0,N}$ for $N>1$ is $G(\A_{\rm fin})$-invariant, so is $V_{1/2}'$. Since $G(\Q_p)$ acts trivially on $V_{1/2,p}/\mathcal{D}_p$, it follows that $G(\Q_p)$ acts trivially on
\begin{equation}\label{V12ppropeq5}
 V_{1/2}\Big/\bigg(V_{1/2,\infty}\otimes\mathcal{D}_p\otimes\Big(\bigotimes_{p'\neq p}V_{1/2,p'}\Big)\bigg),
\end{equation}
and hence on $V_{1/2}/V_{1/2}'$.
\end{proof}

More important for us than $V'_{1/2}$ will be its subspace
\begin{equation}\label{V12ppropeq6}
 V_{1/2}'':=\bigoplus_{\substack{k\in2\Z\\k\neq0}}U_k=\mathcal{D}_1^{\rm hol}\otimes\bigg(\bigotimes_{p<\infty}V_{1/2,p}\bigg).
\end{equation}
For this space we can prove that $A_{1/2}$ has the following intertwining properties.

\begin{proposition}\label{V12pintprop}
 Let $f\in V_{1/2}''$.
  \begin{enumerate}
  \item For $h\in G(\A_{\rm fin})$, let $f^h(g)=f(gh)$. Then
   \begin{equation}\label{V12pintpropeq1}
    (A_{1/2}f^h)(g)=(A_{1/2}f)(gh)\qquad\text{for }h\in G(\A_{\rm fin}).
   \end{equation}
  \item If $f\in U_k$ with $k\neq0$, then
   \begin{align}
    \label{V12pintpropeq2}A_{1/2}(Hf)&=H(A_{1/2}f),\\
    \label{V12pintpropeq3}A_{1/2}(Rf)&=\begin{cases}
                 R(A_{1/2}f)&\text{if }k\neq-2,\\
                 0&\text{if }k=-2,
                \end{cases}\\
    \label{V12pintpropeq4}A_{1/2}(Lf)&=\begin{cases}
                 L(A_{1/2}f)&\text{if }k\neq2,\\
                 0&\text{if }k=2,
                \end{cases}\\
    \label{V12pintpropeq5}A_{1/2}(\varepsilon_-f)&=\varepsilon_-(A_{1/2}f).
   \end{align}
 \end{enumerate}
\end{proposition}
\begin{proof}
It is obvious from \eqref{V12ppropeq6} that $f^h\in V_{1/2}''$. We may assume that $f=\otimes f_v$ is a pure tensor. Let $g=(g_p)$ and $h=(h_p)$. We may also assume that $f_\infty=f_{1/2, \infty}^{(k)}$ with $k\neq0$, i.e., $f\in U_k$. Letting $s$ go to $1/2$ in \eqref{globaleq5}, we see that for every large enough set $T$ of finite places,
\begin{equation}\label{V12pintpropeq8}
 (A_{1/2}f)\left(g\right)=\frac6{\pi^2}(\tilde A_{1/2,\infty}f_{1/2, \infty}^{(k)})(g_\infty)\bigg(\prod_{p\in T}\:\frac{1-p^{-1}}{1-p^{-2}}(A_{1/2,p}f_p)(g_p)\bigg),
\end{equation}
where $\tilde A_{1/2,\infty}f_{1/2, \infty}^{(k)}$ is the function defined in \eqref{intrealeq10}. In this form the intertwining property \eqref{V12pintpropeq1} follows because each $A_{1/2,p}$ is an intertwining operator; we just have to choose a set $T$ large enough so that it works for both $f$ and~$f^h$. Properties \eqref{V12pintpropeq2}--\eqref{V12pintpropeq5} follow from \eqref{localreallemma2eq2}--\eqref{localreallemma2eq5}.
\end{proof}

For later use, we note that \eqref{V12pintpropeq8}, in conjunction with \eqref{localreallemma2eq1}, gives the formula
\begin{equation}\label{V12pintpropeq9}
 (A_{1/2}f)\left(g\right)=-\frac6{\pi|k|}f_{-1/2, \infty}^{(k)}(g_\infty)\bigg(\prod_{p\in T}\:\frac{1}{1+p^{-1}}(A_{1/2,p}f_p)(g_p)\bigg),
\end{equation}
whenever $f=\otimes f_v\in V_{1/2}$ with $f_\infty=f_{1/2, \infty}^{(k)}$ and $T$ is such that $f_p=f_{1/2,p}^{\rm sph}$ and $g_p\in K_p$ for $p\notin T$.
\section{Whittaker integrals} \label{Sect:Whittaker}
In this section we study local and global Whittaker integrals, proceeding analogously to the study of the intertwining operators in the previous section.
\subsection{Non-archimedean case}\label{whitpadic}
Let $F$ be a non-archimedean local field of characteristic zero. We use the same notations as in Sect.~\ref{interpadic}. We fix a character $\psi$ of $F$ of conductor $\OF$. For $\alpha\in F^\times$, we let $\psi^\alpha(x)=\psi(\alpha x)$. For $f\in V_s$, we define the \emph{$\psi^\alpha$-Whittaker function} associated to $f$ by
\begin{equation}\label{localpadicWeq2}
 (W_s^\alpha f)(g):=\lim_{n\to\infty}\int\limits_{\p^{-n}}f(\begin{bsmallmatrix}&-1\\1\end{bsmallmatrix}\begin{bsmallmatrix}1&b\\&1\end{bsmallmatrix}g)\psi(-\alpha b)\,db.
\end{equation}
The sequence given by the integrals stabilizes, and hence the limit exists for all $s$. Provided that ${\rm Re}(s)>0$, it is easy to see that
\begin{equation}\label{localpadicWeq1}
 (W_s^\alpha f)\left(g\right)=\int\limits_Ff(\begin{bsmallmatrix}&-1\\1\end{bsmallmatrix}\begin{bsmallmatrix}1&b\\&1\end{bsmallmatrix}g)\psi(-\alpha b)\,db.
\end{equation}
It follows from \eqref{localpadicWeq2} that, for any $s\in\C$,
\begin{equation}\label{localpadicWeq10}
 (W_s^\alpha f)(\mat{y}{}{}{1}g)=|y|^{1/2-s}\,(W_s^{\alpha y}f)(g)
\end{equation}
for all $g\in G$ and $y,\alpha\in F^\times$.

\begin{lemma}\label{calphaplemma}
 Suppose that $q^{2s}\neq1$. Let $\alpha,y\in F^\times$.
 \begin{enumerate}
  \item If $f_s^{\rm sph}\in V_s$ is the normalized spherical vector, then
   \begin{equation}\label{calphaplemmaeq3}
    (W_s^\alpha f_s^{\rm sph})(\mat{y}{}{}{1})
    =\begin{cases}
      \displaystyle\frac{(|\alpha|^{2s}|y|^{1/2+s}-q^{2s}|y|^{1/2-s})(1-q^{-2s-1})}{1-q^{2s}}&\text{if }v(\alpha y)\geq0,\\
      0&\text{if }v(\alpha y)<0.
     \end{cases}
   \end{equation}   
  \item If $f_s^{\rm St}\in V_s$ is the Steinberg vector defined in \eqref{localpadiceq4}, then
   \begin{equation}\label{calphaplemmaeq4}
    (W_s^\alpha f_s^{\rm St})(\mat{y}{}{}{1})
    =\begin{cases}
      \displaystyle\frac{(1-q^{-2s-1})|\alpha|^{2s}|y|^{1/2+s}-(1-q^{2s-1})|y|^{1/2-s}}{1-q^{2s}}&\text{if }v(\alpha y)\geq0,\\
      0&\text{if }v(\alpha y)<0.
     \end{cases}
   \end{equation}   
 \end{enumerate}
\end{lemma}
\begin{proof}
In view of \eqref{localpadicWeq10}, we may assume that $y=1$. Assume that $f\in V_s$ is $\Gamma_0(\p)$-invariant. Assuming $v(\alpha)\geq0$, a straightforward calculation shows that
\begin{equation}\label{calphaplemmaeq2}
 (W_s^\alpha f)(1)=f(\mat{}{-1}{1}{})+f(1)\,\frac{|\alpha|^{2s}(1-q^{-2s-1})+q^{-1}-1}{1-q^{2s}}.
\end{equation}
Equation \eqref{calphaplemmaeq3} follows by setting $f(1)=f(\mat{}{-1}{1}{})=1$ in \eqref{calphaplemmaeq2}. Equation \eqref{calphaplemmaeq4} follows by setting $f(1)=1$ and $f(\mat{}{-1}{1}{})=-q^{-1}$ in \eqref{calphaplemmaeq2}.
\end{proof}

For $s=1/2$ and $v(\alpha y)\geq0$, the expressions in \eqref{calphaplemmaeq3} and \eqref{calphaplemmaeq4} simplify as follows,
\begin{align}
 \label{localpadicWeq3}(W_{1/2}^\alpha f_{1/2}^{\rm sph})(\mat{y}{}{}{1})&=(1-q^{-1}|\alpha y|)(1+q^{-1}),\\
 \label{localpadicWeq4}(W_{1/2}^\alpha f_{1/2}^{\rm St})(\mat{y}{}{}{1})&=-q^{-1}|\alpha y|(1+q^{-1}).
\end{align}
\subsection{Archimedean case}\label{whitarch}
For $s\in\C$, let $V_s$ be the $\H_\infty$-module considered in Sect.~\ref{interarch}. For $f\in V_s$ and $\alpha\in\R^\times$, we consider the Whittaker function
\begin{equation}\label{localrealWeq1}
 (W_s^\alpha f)(g)=\int\limits_\R f(\begin{bsmallmatrix}&-1\\1\end{bsmallmatrix}\begin{bsmallmatrix}1&b\\&1\end{bsmallmatrix}g)e^{2\pi i\alpha b}\,db.
\end{equation}
We would like to evaluate these integrals if $f=f_s^{(k)}$ is the weight-$k$ function given in \eqref{localrealeq2}. As in the proof of \cite[Thm.~3.7.1]{Bump1997} we find
\begin{equation}\label{localrealWeq2}
 (W_s^\alpha f_s^{(k)})\left(\mat{1}{x}{}{1}\mat{y}{}{}{1}\right)=e^{-2\pi i\alpha x}\,|y|^{1/2-s}\int\limits_\R\bigg(\frac{1}{b^2+1}\bigg)^{s+1/2}\bigg(\frac{b-i}{\sqrt{b^2+1}}\bigg)^ke^{2\pi i\alpha by}\,db
\end{equation}
for ${\rm Re}(s)>0$. Using \cite[3.384.9]{GradshteynRyzhik2015}, we get
\begin{equation}\label{localrealWeq10}
 (W_s^\alpha f_s^{(k)})\left(\mat{1}{x}{}{1}\mat{y}{}{}{1}\right)=e^{-2\pi i\alpha x}\frac{(-1)^{k/2}|\alpha|^{s-1/2}\pi^{s+1/2}}{\Gamma(s+1/2-\sgn(\alpha y)k/2)}W_{-\sgn(\alpha y)k/2,-s}(4\pi|\alpha y|),
\end{equation}
where $W_{*,*}$ is the Whittaker function defined in \cite[9.220]{GradshteynRyzhik2015}. For fixed non-zero $z$, the functions $W_{\kappa,\mu}(z)$ are entire functions of $\kappa$ and $\mu$. It follows that $(W_s^\alpha f_s^{(k)})(g)$ admits analytic continuation to an entire function of $s$.

For $k\geq2$ and $s=\frac{k-1}2$, by \eqref{localrealWeq10} and \cite[13.18.2]{NIST:DLMF},
\begin{equation}\label{localrealWeq5}
 (W_{(k-1)/2}^\alpha f_{(k-1)/2}^{(k)})\left(\mat{1}{x}{}{1}\mat{y}{}{}{1}\right)=
  \begin{cases}
   0&\text{if }\alpha y>0,\\[1ex]
   \displaystyle\frac{(2\pi i)^k}{(k-1)!}|y|^{k/2}|\alpha|^{k-1}e^{-2\pi i\alpha(x+iy)}&\text{if }\alpha y<0.
  \end{cases}
\end{equation}
\subsection{Global case}\label{globalWsec}
We now work over $\Q$, using the same global setup as in Sect.~\ref{globalsec}. Let $\psi$ be the character of $\Q\backslash\A$ defined in Tate's thesis; it has the property that $\psi(x)=\prod\psi_v(x_v)$ with $\psi_\infty(x_\infty)=e^{-2\pi ix_\infty}$. For a finite prime $p$, the character $\psi_p$ of $\Q_p$ is trivial on $\Z_p$ but not on $p^{-1}\Z_p$.

For $f\in V_s$ and $\alpha\in\Q^\times$ we consider the global Whittaker function
\begin{equation}\label{globalWeq1}
 (W_s^\alpha f)(g)=\int\limits_\A f(\begin{bsmallmatrix}&-1\\1\end{bsmallmatrix}\begin{bsmallmatrix}1&b\\&1\end{bsmallmatrix}g)\psi(-\alpha b)\,db.
\end{equation}
To investigate convergence, assume that $f$ corresponds to a pure tensor $\otimes f_v$, and that $g=(g_v)_v$. Then
\begin{equation}\label{globalWeq2}
 (W_s^\alpha f)(g)=\prod_v\:\int\limits_{\Q_v} f_v(\begin{bsmallmatrix}&-1\\1\end{bsmallmatrix}\begin{bsmallmatrix}1&b\\&1\end{bsmallmatrix}g_v)\psi_v(-\alpha b)\,db=\prod_v (W_{s,v}^\alpha f_v)(g_v),
\end{equation}
with the local Whittaker functions $W_{s,v}^\alpha f_v$ defined in \eqref{localpadicWeq1} and \eqref{localrealWeq1}. To ease notation, we will sometimes write $W^\alpha$ instead of $W_{s,v}^\alpha$, if the context is clear.

Let $T$ be a finite set of finite places such that for primes $p\notin T$ the function $f_p$ is the normalized spherical vector, $g_p\in K_p$, and $v_p(\alpha)=0$. Then
\begin{align}\label{globalWeq13}
 (W_s^\alpha f)(g)&=\bigg(\prod_{v\in T\cup\{\infty\}}(W_{s,v}^\alpha f_v)(g_v)\bigg)\bigg(\prod_{p\notin T}(W_{s,p}^\alpha f_p)(1)\bigg)\nonumber\\
 &\stackrel{\eqref{calphaplemmaeq3}}{=}\bigg(\prod_{v\in T\cup\{\infty\}}(W_{s,v}^\alpha f_v)(g_v)\bigg)\bigg(\prod_{p\notin T}(1-p^{-2s-1})\bigg)\nonumber\\
 &=\frac1{\zeta(2s+1)}\,(W_{s,\infty}^\alpha f_\infty)(g_\infty)\bigg(\prod_{p\in T}\frac1{1-p^{-2s-1}}(W_{s,p}^\alpha f_p)(g_p)\bigg).
\end{align}
It follows that \eqref{globalWeq1} converges for ${\rm Re}(s)>0$.

\begin{lemma}\label{alphalatticelemma}
 Let $\alpha\in\Q^\times$ and $f\in V_s$. For fixed $g=(g_p)\in G(\A)$, there exists an integer $M>0$ such that
 \begin{equation}\label{alphalatticelemmaeq1}
  (W_s^\alpha f)(g)=0\qquad\text{if }\alpha\notin M^{-1}\Z.
 \end{equation}
 The integer $M$ depends only on the right-invariance properties of $f$ under the groups $G(\Z_p)$, and on the $g_p$ for $p<\infty$. We may choose $M$ such that it is divisible only by those primes $p$ for which $f$ is not $G(\Z_p)$-invariant or $g_p\notin G(\Z_p)$.
\end{lemma}
\begin{proof}
We may assume that $f=\otimes f_v$ is a pure tensor. Recall from \eqref{globalWeq2} that $(W_s^\alpha f)(g)=\prod_v(W^\alpha f_v)(g_v)$. For a prime $p$, there exists a positive integer $m_p$ such that for $x\in p^{m_p}\Z_p$
\begin{equation}\label{alphalatticelemmaeq2}
 (W^\alpha f_p)(g_p)=(W^\alpha f_p)(\mat{1}{x}{}{1}g_p)=\psi_p(\alpha x)(W^\alpha f_p)(g_p).
\end{equation}
It follows that $(W^\alpha f_p)(g_p)=0$ if $\alpha\notin p^{-m_p}\Z_p$. If $f_p$ is spherical and $g_p\in G(\Z_p)$, we may choose $m_p=0$. Then \eqref{alphalatticelemmaeq1} holds with $M=\prod_pp^{m_p}$. This concludes the proof.
\end{proof}

Using a common notation, we set, for a complex number $t$ and a positive integer $n$,
\begin{equation}\label{sigmatdefeq}
 \sigma_t(n)=\sum_{d\mid n}d^t.
\end{equation}
The sum is understood to be over the positive divisors of $n$. We write $\sigma$ for $\sigma_1$.

\begin{lemma}\label{Walphasigmalemma}
 Let $\alpha$ be a non-zero integer.
 \begin{enumerate}
  \item For ${\rm Re}(s)>0$,
   \begin{equation}\label{Walphasigmalemmaeq1}
    \prod_{p<\infty}(W^\alpha f_{s,p}^{\rm sph})(1)=\frac{\sigma_{2s}(|\alpha|)}{\zeta(2s+1)|\alpha|^{2s}}.
   \end{equation}
  \item For $s=1/2$ and a positive, square-free integer $N$,
   \begin{equation}\label{Walphasigmalemmaeq2}
    \bigg(\prod_{p\mid N}(W^\alpha f_{1/2,p}^{\rm St})(1)\bigg)\bigg(\prod_{p\nmid N}(W^\alpha f_{1/2,p}^{\rm sph})(1)\bigg)=\frac{\sigma(n')\mu(N)}{\zeta(2)|\alpha|\varphi(N)},
   \end{equation}
   where $\mu$ is the M\"obius function, $\varphi$ is Euler's function, $\sigma=\sigma_1$, and $n'=\prod_{p\nmid N}p^{v_p(\alpha)}$ is the part of $|\alpha|$ that is relatively prime to $N$.
 \end{enumerate}
\end{lemma}
\begin{proof}
By Lemma~\ref{calphaplemma},
\begin{align}\label{Walphasigmalemmaeq3}
  &\bigg(\prod_{p\mid N}(W^\alpha f_{s,p}^{\rm St})(1)\bigg)\bigg(\prod_{p\nmid N}(W^\alpha f_{s,p}^{\rm sph})(1)\bigg)\nonumber\\
  &\qquad=\bigg(\prod_{p\mid N}\frac{(1-p^{-2s-1})|\alpha|_p^{2s}-(1-p^{2s-1})}{1-p^{2s}}\bigg)\bigg(\prod_{p\nmid N}\frac{(|\alpha|_p^{2s}-p^{2s})(1-p^{-2s-1})}{1-p^{2s}}\bigg).
\end{align}
Setting $N=1$ gives
\begin{align}\label{Walphasigmalemmaeq4}
  \prod_p (W^\alpha f_{s,p}^{\rm sph})(1)&=\frac1{\zeta(2s+1)}\prod_p\frac{|\alpha|_p^{2s}-p^{2s}}{1-p^{2s}}\nonumber\\
  &=\frac1{\zeta(2s+1)|\alpha|_\infty^{2s}}\,\prod_p\frac{1-p^{2s(1+v_p(\alpha))}}{1-p^{2s}}\nonumber\\
  &=\frac1{\zeta(2s+1)|\alpha|_\infty^{2s}}\,\prod_p\Big(1+p^{2s}+\ldots+(p^{2s})^{v_p(\alpha)}\Big).
\end{align}
This proves \eqref{Walphasigmalemmaeq1}. Setting $s=1/2$ in \eqref{Walphasigmalemmaeq3} gives
\begin{align}\label{Walphasigmalemmaeq5}
  \bigg(\prod_{p\mid N}(W^\alpha f_{1/2,p}^{\rm St})(1)\bigg)\bigg(\prod_{p\nmid N}(W^\alpha f_{1/2,p}^{\rm sph})(1)\bigg)&=\bigg(\prod_{p\mid N}\frac{(1-p^{-2})|\alpha|_p}{1-p}\bigg)\bigg(\prod_{p\nmid N}\frac{(|\alpha|_p-p)(1-p^{-2})}{1-p}\bigg)\nonumber\\
  &=\frac1{\zeta(2)|\alpha|_\infty}\bigg(\prod_{p\mid N}\frac{1}{1-p}\bigg)\bigg(\prod_{p\nmid N}\frac{1-p^{1+v_p(\alpha)}}{1-p}\bigg)\nonumber\\
  &=\frac{\mu(N)}{\zeta(2)|\alpha|_\infty}\bigg(\prod_{p\mid N}\frac{1}{p-1}\bigg)\sigma(n').
\end{align}
This proves \eqref{Walphasigmalemmaeq2}, because $\varphi(N)=\prod_{p|N}(p-1)$.
\end{proof}
\section{Eisenstein series}\label{Sect:Eisenstein series}
In this section we prove our main results for Eisenstein series without character. The preparations from Sects.~\ref{Sect:intertwining} and~\ref{Sect:Whittaker} allow us to make the connection between adelic and classical Eisenstein series. Theorem~\ref{E2reptheorem} identifies the global representation generated by the classical $E_2$.
\subsection{Fourier expansion}\label{Fouriersec}
As in Sect.~\ref{globalsec}, we consider the global $\H$-module $V_s$. For any $f\in V_s$, define the Eisenstein series
\begin{equation}\label{Fouriereq1}
 E(g,f)=\sum_{\gamma\in B(\Q)\backslash G(\Q)}f(\gamma g),\qquad g\in G(\A).
\end{equation}
By \cite[Prop.~3.7.2]{Bump1997}, the sum converges absolutely if ${\rm Re}(s)>1/2$. Under this assumption, $E(\cdot,f)$ is an automorphic form on $G(\A)$. In order to analytically continue the Eisenstein series to other values of $s$, the key is to consider the Fourier expansion and analytically continue each piece. Even though this is part of a general theory, we briefly recall the main steps for our rather simple situation. For $\alpha\in\Q$, the $\alpha$-th Fourier coefficient of $E(g,f)$ is defined by
\begin{equation}\label{Fouriereq2}
 c_\alpha(g,f)=\int\limits_{\Q\backslash\A}E(\mat{1}{b}{}{1}g,f)\psi(-\alpha b)\,db.
\end{equation}
The Fourier expansion of the Eisenstein series is
\begin{equation}\label{Fouriereq3}
 E(g,f)=\sum_{\alpha\in\Q}c_\alpha(g,f).
\end{equation}
To calculate the $c_\alpha(g,f)$, we use that, by the Bruhat decomposition, a set of representatives for $B(\Q)\backslash G(\Q)$ is given by $1$ and $\mat{}{-1}{1}{}\matb{1}{x}{}{1}$, $x\in\Q$. Substituting
\begin{equation}\label{Fouriereq4}
 E(g,f)=f(g)+\sum_{x\in\Q}f(\mat{}{-1}{1}{}\matb{1}{x}{}{1}g)
\end{equation}
into \eqref{Fouriereq2} gives
\begin{equation}\label{Fouriereq5}
 c_\alpha(g,f)=\int\limits_{\A}f(\mat{}{-1}{1}{}\mat{1}{b}{}{1}g)\psi(-\alpha b)\,db
\end{equation}
for $\alpha\neq0$, and
\begin{equation}\label{Fouriereq6}
 c_0(g,f)=f(g)+\int\limits_{\A}f(\mat{}{-1}{1}{}\mat{1}{b}{}{1}g)\,db.
\end{equation}
The integrals in \eqref{Fouriereq5} were analyzed in Sect.~\ref{globalWsec}, where we called them $(W_s^\alpha f)(g)$ (see \eqref{globalWeq1}), and found to be convergent for ${\rm Re}(s)>0$. The integrals in \eqref{Fouriereq6} were analyzed in Sect.~\ref{globalsec}, where we called them $(A_sf)(g)$ (see \eqref{globaleq2}), and found to be convergent for ${\rm Re}(s)>1/2$.
\begin{lemma}\label{Fourierlemma1}
 Assume that ${\rm Re}(s)>1/2$ and $f\in V_s$. For fixed $g=(g_p)\in G(\A)$, there exists an integer $M>0$ such that
 \begin{equation}\label{Fourierlemma1eq1}
  E(g,f)=f(g)+(A_sf)(g)+\sum_{\substack{\alpha\in M^{-1}\Z\\\alpha\neq0}}(W_s^\alpha f)(g).
 \end{equation}
 The integer $M$ depends only on right-invariance properties of $f$ under the groups $G(\Z_p)$, and on the $g_p$ for $p<\infty$. We may choose $M$ such that it is divisible only by those primes $p$ for which $f$ is not $G(\Z_p)$-invariant or $g_p\notin G(\Z_p)$.
\end{lemma}
\begin{proof}
We may assume that $f=\otimes f_v$ is a pure tensor. By \eqref{Fouriereq3}, \eqref{Fouriereq5}, \eqref{Fouriereq6} we have the expansion \eqref{Fourierlemma1eq1} with the summation being over $\alpha\in\Q^\times$. By Lemma~\ref{alphalatticelemma} the summation can be restricted to $\alpha\in M^{-1}\Z$ as asserted.
\end{proof}

\begin{lemma}\label{Fourierlemma2}
 Assume that $f_s\in V_s$ is a flat section. For ${\rm Re}(s)>1/2$ and fixed $g=(g_p)\in G(\A)$, let $M$ be as in Lemma~\ref{Fourierlemma1}, so that
 \begin{equation}\label{Fourierlemma2eq1}
  E(g,f_s)=f_s(g)+(A_sf_s)(g)+\sum_{\substack{\alpha\in M^{-1}\Z\\\alpha\neq0}}(W_s^\alpha f_s)(g).
 \end{equation}
 The term $\sum_\alpha(W_s^\alpha f_s)(g)$ admits an analytic continuation to ${\rm Re}(s)>0$. The term $(A_sf_s)(g)$ admits an analytic continuation to the same region, except possibly for $s=1/2$. Hence $E(g,f_s)$ admits an analytic continuation to ${\rm Re}(s)>0$, $s\neq1/2$.
\end{lemma}
\begin{proof}
By the considerations in Sect.~\ref{globalsec} and Sect.~\ref{V12psec}, we need only prove the statement about $\sum_\alpha (W_s^\alpha f_s)(g)$. For this we only give a rough sketch, since in the cases of interest for us everything will follow from explicit calculations. In general, one may assume that $f_s=\otimes f_{s,v}$ is a pure tensor, and that the archimedean section is the function $f_s^{(k)}$ of weight $k$ defined in \eqref{localrealeq2}, for some even integer $k$. We saw in Sect.~\ref{globalWsec} that each individual $(W_s^\alpha f_s)(g)$ continues analytically to ${\rm Re}(s)>0$. Using the Iwasawa decomposition, we may assume that $g=\mat{y}{}{}{1}$ with $y\in\A^\times$. We then have to show that the series of functions $\sum_{\alpha\in M^{-1}\Z,\,\alpha\neq0}F_\alpha(s)$, where $F_\alpha(s):=(W_s^\alpha f_s)\left(\mat{y}{}{}{1}\right)$, is uniformly convergent on a bounded domain $D$ in ${\rm Re}(s)>0$. The key is that there exists a polynomial $P\in\Q[X]$, independent of $\alpha$ and $s\in D$, such that
\begin{equation}\label{Fourierlemma2eq7}
 \Big|(W_s^\alpha f_s)\left(\mat{y}{}{}{1}\right)\Big|\leq P(|\alpha|_\infty)e^{-2\pi|\alpha y_\infty|_\infty}.
\end{equation}
The exponential comes from the archimedean place, more precisely from an estimate on the classical Whittaker function appearing in \eqref{localrealWeq10} (see \cite[16.3]{WhittakerWatson1962}). To prove that the contribution from the non-archimedean places grows at most polynomially in $\alpha$, one can use the description of the Kirillov model in \cite[Thms.~4.7.2, 4.7.3]{Bump1997}.
\end{proof}

The possible pole at $s=1/2$ in Lemma~\ref{Fourierlemma2} comes from the term $(A_sf_s)(g)$. Let $V_{1/2}'$ and $V_{1/2}''$ be the subspaces of $V_{1/2}$ defined in Sect.~\ref{V12psec}. Then
\begin{equation}\label{intpropseceq3}
 E(g,f):=\lim_{s\to1/2}\,E(g,f_s)
\end{equation}
exists for $f\in V_{1/2}'$ and all $g\in G(\A)$; here, $f_s$ is the unique flat section containing $f$. Evidently,
\begin{equation}\label{intpropseceq4}
 E(g,f)=f(g)+(A_{1/2}f)(g)+\sum_{\substack{\alpha\in M^{-1}\Z\\\alpha\neq0}}(W_{1/2}^\alpha f)(g),
\end{equation}
with $A_{1/2}f$ as in \eqref{V12peq1}.

It follows directly from the definition \eqref{Fouriereq1} that the map $f\mapsto E(\cdot,f)$ from $V_s$ to the space of automorphic forms is intertwining (i.e., a homomorphism of $\H$-modules) if ${\rm Re}(s)>1/2$. This is less obvious for the analytically continued Eisenstein series, and in fact it is not true for $s=1/2$. The following result clarifies which intertwining properties are retained. Recall the definition of the space $U_k$ in \eqref{Ukdefeq}.
\begin{lemma}\label{intpropprop}
 For $f\in V_{1/2}''$, let $E(g,f)$ be as defined in \eqref{intpropseceq3}.
 \begin{enumerate}
  \item For $h\in G(\A_{\rm fin})$, let $f^h(g)=f(gh)$. Then
   \begin{equation}\label{intproppropeq1}
    E(g,f^h)=E(gh,f)\qquad\text{for }h\in G(\A_{\rm fin}).
   \end{equation}
  \item If $f\in U_k$ with $k\neq0$, then
   \begin{align}
    \label{intproppropeq2}E(\cdot,Hf)&=H\,E(\cdot,f),\\
    \label{intproppropeq3}E(\cdot,Rf)&=\begin{cases}
                 R\,E(\cdot,f)&\text{if }k\neq-2,\\
                 0&\text{if }k=-2,
                \end{cases}\\
    \label{intproppropeq4}E(\cdot,Lf)&=\begin{cases}
                 L\,E(\cdot,f)&\text{if }k\neq2,\\
                 0&\text{if }k=2,
                \end{cases}\\
    \label{intproppropeq5}E(\cdot,\varepsilon_-f)&=\varepsilon_-\,E(\cdot,f).
   \end{align}
 \end{enumerate}
\end{lemma}
\begin{proof}
This follows from \eqref{intpropseceq4}. The intertwining properties for the first and the third term on the right hand side are clear. (Observe that $Rf=0$ for $f\in U_{-2}$ and $Lf=0$ for $f\in U_2$.) The properties for the second term follow from Proposition~\ref{V12pintprop}.
\end{proof}
\subsection{The Eisenstein series \texorpdfstring{$E_k$}{}}
Recall that the classical Eisenstein series $E_k$, defined in \eqref{introeq2} for an even integer $k\geq4$, are modular forms of weight $k$ with a Fourier expansion
\begin{equation}\label{Ekeq1}
 E_k(z)=1+\frac{(2\pi i)^k}{\zeta(k)(k-1)!}\sum_{n=1}^\infty\sigma_{k-1}(n)e^{2\pi inz}.
\end{equation}
(See \cite[Sect.~1.1]{DiamondShurman2005}.) In addition, there is the Eisenstein series $E_2$, defined by the conditionally convergent series \eqref{introeq4}. It is a non-holomorphic modular form of weight $2$ with Fourier expansion
\begin{equation}\label{EKeq1b}
 E_2(z)=1-\frac3{\pi y}-24\sum_{n=1}^\infty\sigma(n)e^{2\pi inz}.
\end{equation}
(See \cite[Sect.~1.2]{DiamondShurman2005}.) Theorem~\ref{Ektheorem} below will explain the adelic origin of these Eisenstein series.

First, we make some comments about the correspondence between automorphic forms on $G(\A)$ and functions on the upper half plane $\mathbb{H}$. Suppose an automorphic form $\Phi$ on $G(\A)$ is invariant under the adelic center. Assume also that it is right-invariant under $\prod_{p<\infty}K'_p$, where $K'_p$ is an open-compact subgroup of $K_p$, with $K'_p=K_p$ for almost all $p$, and the determinant map $K'_p\to\Z_p^\times$ is surjective for all $p$. Then, by strong approximation, $\Phi$ is determined by its values on $\GL(2,\R)^+$. Assume also that $\Phi$ has weight $k$ for some integer $k$, i.e., $\Phi(gr(\theta))=e^{ik\theta}\Phi(g)$ for all $g\in G(\A)$ and $\theta\in\R$. Then $\Phi$ is determined by its values on elements of the form $\mat{1}{x}{}{1}\mat{y}{}{}{1}$ with $x,y\in\R$ and $y>0$. Now, whenever we have a weight-$k$ function $\Phi$ on $\GL(2,\R)^+$ invariant under the archimedean center, we can define a function $F$ on the upper half plane by
\begin{equation}\label{Fouriereq9}
 F(z)=\det(g)^{-k/2}j(g,i)^k\Phi(g),\qquad\text{where }g\in \GL(2,\R)^+\text{ is such that }gi=z.
\end{equation}
Here, $j(g,z)=cz+d$ for $g=\mat{a}{b}{c}{d}$, as usual. We can take a specific $g$, namely,
\begin{equation}\label{Fouriereq10}
 F(z)=y^{-k/2}\Phi(\mat{1}{x}{}{1}\mat{y}{}{}{1}),\qquad z=x+iy.
\end{equation}
The $F$ thus defined transforms like a modular form of weight $k$ under $\Gamma=\SL(2,\Q)\cap\prod_{p<\infty}K'_p$, i.e., $F|_k\gamma=F$ for $\gamma\in\Gamma$.

Conversely, starting with a function $F$ on $\mathbb{H}$ with this transformation property, we can define a weight-$k$ function $\Phi$ on $G(\A$) such that \eqref{Fouriereq9} holds. If $F$ is sufficiently regular (e.g., a modular form), then $\Phi$ is an automorphic form. We call it the automorphic form corresponding to $F$.

\begin{theorem}\label{Ektheorem}
 Let $k\geq2$ be an even integer. Let $f_k\in V_{(k-1)/2}$ be the following pure tensor,
 \begin{equation}\label{Ektheoremeq0}
  f_k=f_{(k-1)/2,\infty}^{(k)}\otimes\Big(\otimes_{p<\infty}f_{(k-1)/2,p}^{\rm sph}\Big).
 \end{equation}
 Then $E(\cdot,f_k)$ is the automorphic form corresponding to $E_k$.
\end{theorem}
\begin{proof}
We first assume that $k\geq4$. By Lemma~\ref{Fourierlemma1}
\begin{equation}\label{Ekeq2}
  E(g,f_k)=f_k(g)+(A_{(k-1)/2}f_k)(g)+\sum_{\substack{\alpha\in\Z\\\alpha\neq0}}(W_{(k-1)/2}^\alpha f_k)(g).
\end{equation}
It follows from \eqref{localrealeq6} and \eqref{globaleq5} that $A_{(k-1)/2}f_k=0$. Applying \eqref{Fouriereq10} to our function \eqref{Ekeq2}, we see that the corresponding function on the upper half plane is given by
\begin{equation}\label{Ekeq3}
 F(z)=1+y^{-k/2}\sum_{\substack{\alpha\in\Z\\\alpha\neq0}}(W_{(k-1)/2}^\alpha f_k)(\mat{1}{x}{}{1}\mat{y}{}{}{1}).
\end{equation}
By \eqref{localrealWeq5} and \eqref{Walphasigmalemmaeq1}, for $\alpha\neq0$,
\begin{align}\label{Ekeq4}
 (W_{(k-1)/2}^\alpha f_k)(\mat{1}{x}{}{1}\mat{y}{}{}{1})&=(W^\alpha f_{(k-1)/2,\infty}^{(k)})(\mat{1}{x}{}{1}\mat{y}{}{}{1})\prod_{p<\infty}(W^\alpha f_{(k-1)/2,p}^{\rm sph})(1)\nonumber\\
 &=\begin{cases}
     0&\text{if }\alpha>0,\\
     \displaystyle y^{k/2}\frac{(2\pi i)^k}{\zeta(k)(k-1)!}\sigma_{k-1}(|\alpha|)e^{-2\pi i\alpha(x+iy)}&\text{if }\alpha<0,
   \end{cases}
\end{align}
Hence, writing $n=-\alpha$, we see that $F$ equals the function $E_k$ in \eqref{Ekeq1}.

Now assume that $k=2$. In this case the argument is very similar, but instead of Lemma~\ref{Fourierlemma1} we use the analytically continued version \eqref{intpropseceq4}. The main difference for $k=2$ is that the intertwining operator is non-zero; using \eqref{V12pintpropeq9} with $T=\emptyset$, we get
\begin{equation}\label{Ekeq5}
 (A_{1/2}f_2)(\mat{1}{x}{}{1}\mat{y}{}{}{1})=-\frac3\pi.
\end{equation}
Equation \eqref{Ekeq4} simplifies for $\alpha<0$ to $-24y\sigma(|\alpha|)e^{-2\pi i\alpha(x+iy)}$. We see that the corresponding function on the upper half plane is precisely the classical $E_2$ given in \eqref{EKeq1b}.
\end{proof}
\subsection{Eisenstein series of weight \texorpdfstring{$2$}{} with level}\label{E2Nsec}
Let $N$ be a positive integer. Starting from the non-holomorphic Eisenstein series $E_2$ in \eqref{EKeq1b}, one forms the function
\begin{equation}\label{Eisleveleq1}
 \tilde E_{2,N}(z)=E_2(z)-NE_2(Nz).
\end{equation}
The $\frac1y$-terms cancel out, so that one obtains a holomorphic modular form of weight $2$ with respect to $\Gamma_0(N)$. (See \cite[Sect.~1.2]{DiamondShurman2005}.) The Fourier expansion of $\tilde E_{2,N}$ is
\begin{equation}\label{Eisleveleq2}
 \tilde E_{2,N}(z)=1-N-24\sum_{n=1}^\infty a_ne^{2\pi inz},
\end{equation}
where
\begin{equation}\label{Eisleveleq3}
 a_n=\begin{cases}
                  \sigma(n)&\text{if }N\nmid n,\\
                  \sigma(n)-N\sigma(n/N)&\text{if }N\mid n.
                 \end{cases}
\end{equation}
The adelic origin of this function is not difficult to determine. Let $f_2\in V_{1/2}$ be the function from Theorem~\ref{Ektheorem}, so that $E(\cdot,f_2)$ is the automorphic form corresponding to $E_2$. A straightforward calculation shows that
\begin{equation}\label{Eisleveleq4}
 g\mapsto E(g\mat{1}{}{}{N_{\rm fin}},f_2),\qquad\text{where }N_{\rm fin}=(N,N,N,\ldots)\in\A_{\rm fin}^\times,
\end{equation}
is the automorphic form corresponding to $NE_2(Nz)$. Since the Eisenstein series is $\H_{\rm fin}$-intertwining by Lemma~\ref{intpropprop} i), it follows that $E(\cdot,\tilde f_{2,N})$, where $\tilde f_{2,N}=f_2-\mat{1}{}{}{N_{\rm fin}}f_2$, is the automorphic form corresponding to $\tilde{E}_{2,N}$. 

Now assume that $N$ is square-free. From an adelic point of view, instead of the functions $\tilde f_{2,N}$, it is more natural to consider
\begin{equation}\label{Eisleveleq6}
 f_{2,N}:=f_{1/2, \infty}^{(2)}\otimes\Big(\otimes_{p\mid N}f_{1/2,p}^{\rm St}\Big)\otimes\Big(\otimes_{p\nmid N}f_{1/2,p}^{\rm sph}\Big).
\end{equation}
In the following theorem $\mu$ denotes the M\"obius function and $\varphi$ denotes Euler's function.
\begin{theorem}\label{tildeENtheorem}
 Let $N>1$ be a squarefree integer and $f_{2,N}\in V_{1/2}$ be as in \eqref{Eisleveleq6}. Then the function on the upper half plane corresponding to $E(\cdot,f_{2,N})$ is given by
 \begin{equation}\label{tildeENtheoremeq1}
  E_{2,N}(z)=1-24\frac{\mu(N)}{\varphi(N)}\sum_{n=1}^\infty\sigma(n')e^{2\pi inz},
 \end{equation}
 where $n'=\prod_{p\nmid N}p^{v_p(n)}$ is the part of $n$ relatively prime to $N$.
 It is a holomorphic modular form of weight $2$ with respect to $\Gamma_0(N)$.
\end{theorem}
\begin{proof}
Proceeding as in Theorem~\ref{Ektheorem}, we have
\begin{equation}\label{tildeENtheoremeq2}
  E(g,f_{2,N})=f_{2,N}(g)+(A_{1/2}f_{2,N})(g)+\sum_{\substack{\alpha\in\Z\\\alpha\neq0}}(W_{1/2}^\alpha f_{2,N})(g),
\end{equation}
and $E_{2,N}(z)=y^{-1}E(\mat{1}{x}{}{1}\mat{y}{}{}{1}, f_{2, N})$ for $z=x+iy$. Looking at \eqref{globaleq5}, we see that $(A_{1/2}f_{2,N})(g)=0$, because there is at least one finite place $p$ for which $f_{2,N,p}$ lies in the kernel of $A_{1/2,p}$ (observe the hypothesis $N>1$). The Whittaker functions are calculated as in \eqref{Ekeq4}, using \eqref{localrealWeq5} and \eqref{Walphasigmalemmaeq2}. The assertion follows.
\end{proof}

To understand the relationship between $E_{2,N}$ and $\tilde E_{2,N}$, we consider elements of the Hecke algebra $\H_{\rm fin}$. We define three elements of the local Hecke algebra $\H_p$ as follows,
\begin{equation}\label{Eisleveleq7}
 \alpha_p={\rm char}(K_p),\qquad\beta_p={\rm char}(\mat{1}{}{}{p}K_p),\qquad\gamma_p=\frac1{1-p}(\alpha_p-\beta_p),
\end{equation}
where ``${\rm char}$'' means ``characteristic function of''. For a square-free, positive integer $N$, let $B_N,\tilde B_N\in\H_{\rm fin}$ be defined by
\begin{align}
 \label{Eisleveleq10}\tilde B_N&=\bigg(\bigotimes_{p<\infty}\alpha_p\bigg)-\bigg(\bigotimes_{p\mid N}\beta_p\bigg)\otimes\bigg(\bigotimes_{p\nmid N}\alpha_p\bigg),\\
 \label{Eisleveleq11}B_N&=\bigg(\bigotimes_{p\mid N}\gamma_p\bigg)\otimes\bigg(\bigotimes_{p\nmid N}\alpha_p\bigg).
\end{align}
It follows from the definitions of $\tilde f_{2,N}$ and $f_{2,N}$ above that
\begin{equation}\label{Eisleveleq12}
 \tilde B_Nf_2=\tilde f_{2,N},\qquad B_Nf_2=f_{2,N},
\end{equation}
where as before $f_2\in V_{1/2}$ is the function from Theorem~\ref{Ektheorem}. For the second equality, note that $\gamma_p(f_{1/2,p}^{\rm sph})=f_{1/2,p}^{\rm St}$ by \eqref{localpadiceq7}.
\begin{proposition}\label{tildeENrelationprop}
 For all square-free integers $N>1$,
 \begin{equation}\label{tildeENrelationpropeq1}
  E_{2,N}=-\frac{1}{\varphi(N)}\sum_{\substack{M\mid N\\M\neq1}}\mu(N/M)\tilde E_{2,M},\qquad \tilde E_{2,N}=-\sum_{\substack{M\mid N\\M\neq1}}\varphi(M)E_{2,M}.
 \end{equation}
\end{proposition}
\begin{proof}
We calculate
\begin{equation}\label{Eisleveleq13}
 B_N=\bigg(\prod_{p\mid N}\frac1{1-p}\bigg)\bigg(\bigotimes_{p\mid N}(\alpha_p-\beta_p)\bigg)\otimes\bigg(\bigotimes_{p\nmid N}\alpha_p\bigg)=\frac{\mu(N)}{\varphi(N)}\sum_{M\mid N}\mu(M)X_M,
\end{equation}
where
\begin{equation}\label{Eisleveleq14}
 X_M=\bigg(\bigotimes_{p\mid M}\beta_p\bigg)\otimes\bigg(\bigotimes_{p\nmid M}\alpha_p\bigg).
\end{equation}
Note that $X_M=X_1-\tilde B_M$ by \eqref{Eisleveleq10}. Since $\sum_{M\mid N}\mu(M)=0$ for $N>1$, it follows that
\begin{equation}\label{Eisleveleq15}
 B_N=-\frac{1}{\varphi(N)}\sum_{M\mid N}\mu(N/M)\tilde B_M
\end{equation}
for $N>1$. Let us temporarily redefine $B_1$ to be zero, so that \eqref{Eisleveleq15} also holds for $N=1$. Then, by M\"obius inversion,
\begin{equation}\label{Eisleveleq16}
 \tilde B_N=-\sum_{M\mid N}\varphi(M)B_M.
\end{equation}
We apply both sides of \eqref{Eisleveleq15} and \eqref{Eisleveleq16} to $f_2$, and get from \eqref{Eisleveleq12} that
\begin{equation}\label{Eisleveleq17}
 f_{2,N}=-\frac{1}{\varphi(N)}\sum_{\substack{M\mid N\\M\neq1}}\mu(N/M)\tilde f_{2,M},\qquad \tilde f_{2,N}=-\sum_{\substack{M\mid N\\M\neq1}}\varphi(M)f_{2,M},
\end{equation}
for $N>1$. Now all we have to do is build the adelic Eisenstein series on both sides of these equations. The equality of the adelic Eisenstein series implies the equality of the classical Eisenstein series in \eqref{tildeENrelationpropeq1}.
\end{proof}
\subsection{Global representations generated by Eisenstein series}\label{Sec:globalrepEis}
Some of the results of this section are also contained in \cite{Horinaga2021}. The following, which is more or less well known, is a consequence of Theorem~\ref{Ektheorem}.
\begin{corollary}\label{Ektheoremcor}
 If $k\geq4$ is an even integer, then the $\H$-module $\pi$ generated by the automorphic form corresponding to $E_k$ is irreducible. We have $\pi\cong\bigotimes\pi_v$, with $\pi_\infty=\mathcal{D}_{k-1}^{\rm hol}$, the discrete series representation of lowest weight $k$, and $\pi_p=|\cdot|_p^{(k-1)/2}\times|\cdot|_p^{(1-k)/2}$, an irreducible principal series representation, for all $p<\infty$.
\end{corollary}
\begin{proof}
Let $f_k$ be as in Theorem~\ref{Ektheorem}. Recall from \eqref{realexacteq1} that the archimedean component $f_{(k-1)/2,\infty}^{(k)}$ is the lowest weight vector in the $\H_\infty$-submodule $\mathcal{D}_{k-1}^{\rm hol}$ of $V_{(k-1)/2,\infty}$. All $V_{(k-1)/2,p}$ for $p<\infty$ are irreducible. Hence
\begin{equation}\label{Ektheoremcoreq1}
 W:=\mathcal{D}_{k-1}^{\rm hol}\otimes\Big(\bigotimes_{p<\infty}V_{(k-1)/2,p}\Big)
\end{equation}
is an irreducible $\H$-module containing $f_k$. The map $f\mapsto E(\cdot,f)$ from $W$ to the space of automorphic forms is intertwining, since the summation \eqref{Fouriereq1} is absolutely convergent. As we saw, it is a non-zero map, because $E(\cdot,f_k)$ is the automorphic form corresponding to $E_k$. Hence the map is an isomorphism onto its image, proving the result.
\end{proof}

To prove a similar result for $k=2$ is more difficult, since the map $f\mapsto E(\cdot,f)$ is not quite an intertwining operator for $s=1/2$; see Lemma~\ref{intpropprop}. We require the following preparations.
\begin{lemma}\label{algebraactionlemma3}
 For $i\in\{1,\ldots,n\}$ let $G_i$ be a group acting on a vector space $V_i$. Then $G=G_1\times\ldots\times G_n$ acts on $V:=V_1\otimes\ldots\otimes V_n$. Assume that each $V_i$ is a $G_i$-module of length $2$, with a unique simple submodule $W_i$, and with $V_i/W_i\ncong W_i$. Then $W:=W_1\otimes\ldots\otimes W_n$ is the unique simple submodule of $V$.
\end{lemma}
\begin{proof}
The constituents of $V$ are the modules $U_1\otimes\ldots\otimes U_n$, where each $U_i$ is isomorphic to either $W_i$ or $V_i/W_i$. In particular, the constituents are pairwise non-isomorphic. It is therefore enough to show that any simple $G$-submodule $X$ of $V$ is \emph{isomorphic} to $W$. (Then automatically $X=W$, since otherwise we can construct a composition series $0\subset W\subset X\oplus W\subset\ldots\subset V$, in which the isomorphism class of $W$ would occur twice.)

Hence assume that $X=U_1\otimes\ldots\otimes U_n$ with $U_i$ isomorphic to either $W_i$ or $V_i/W_i$, and suppose we have an injection $\varphi:X\to V$. For $i\in\{1,\ldots,n-1\}$, choose any non-zero $u_i\in U_i$. Then we have an injection of vector spaces
\begin{equation}\label{algebraactionlemma3eq3}
 \alpha:\:U_n\longrightarrow U_1\otimes\ldots\otimes U_{n-1}\otimes U_n,\qquad
 \alpha(u)=u_1\otimes\ldots\otimes u_{n-1}\otimes u.
\end{equation}
The map \eqref{algebraactionlemma3eq3} commutes with the action of $G_n$. Pick some non-zero $u_0\in U_n$, and write
\begin{equation}\label{algebraactionlemma3eq4}
 \varphi(\alpha(u_0))=\sum_{j=1}^mw_j\otimes v_j,\qquad w_j\in V_1\otimes\ldots\otimes V_{n-1},\;v_j\in V_n.
\end{equation}
We may assume $m$ to be minimal, in which case $w_1,\ldots,w_m$ (and also $v_1,\ldots,v_m$) are linearly independent. Choose a linear form $f$ on $V_1\otimes\ldots\otimes V_{n-1}$ such that $f(w_1)=1$ and $f(w_j)=0$ for $j\in\{2,\ldots,m\}$. Define
\begin{equation}\label{algebraactionlemma3eq5}
 \beta:\:(V_1\otimes\ldots\otimes V_{n-1})\otimes V_n\longrightarrow V_n,\qquad\beta(w\otimes v)=f(w)v.
\end{equation}
Then $\beta(\varphi(\alpha(u_0)))=v_1$. In particular, the linear map $\beta\circ\varphi\circ\alpha:U_n\to V_n$ is not zero. It is easy to check that this map commutes with the action of $G_n$ on both sides. Since $W_n$ is the unique irreducible subspace of $V_n$, it follows that $U_n\cong W_n$.

In the same manner one proves $U_i\cong W_i$ as $G_i$-representations for all $i$.
\end{proof}

Let $f_2$ be as in Theorem~\ref{Ektheorem}, i.e., 
\begin{equation}\label{f2defeq}
 f_2=f_{1/2,\infty}^{(2)}\otimes\Big(\otimes_{p<\infty}f_{1/2,p}^{\rm sph}\Big).
\end{equation}
Then $\Phi_2:=E(\cdot,f_2)$ is the automorphic form corresponding to $E_2$. For a prime $p$ we let $\mathcal{D}_p\cong\St_{\GL(2,\Q_p)}$ be the infinite-dimensional irreducible, invariant subspace of $V_{1/2,p}$.

\begin{lemma}\label{HfinPhilemma}
 Any $\H_{\rm fin}$-map
 \begin{equation}\label{HfinPhilemmaeq1}
  \bigotimes_{p<\infty}V_{1/2,p}\longrightarrow\H\Phi_2
 \end{equation}
 whose image contains $R^\ell\Phi_2$ or $\varepsilon_-R^\ell\Phi_2$ for some $\ell\geq0$ is injective.
\end{lemma}
\begin{proof}
Let $\phi:\:\bigotimes_{p<\infty}V_{1/2,p}\longrightarrow\H_{\rm fin}\Phi_2$ be an $\H_{\rm fin}$-map whose image contains $R^\ell\Phi_2$ or $\varepsilon_-R^\ell\Phi_2$ for some $\ell\geq0$. By composing with $\varepsilon_-$ if necessary, we may assume that the image of $\phi$ contains $R^\ell\Phi_2$. Let $\rho=\otimes\rho_p$ be the representation of $\H_{\rm fin}$ on $\bigotimes_{p<\infty}V_{1/2,p}$. Let $f\in\bigotimes_{p<\infty}V_{1/2,p}$ be a vector with $\phi(f)=R^\ell\Phi_2$. Then
\begin{equation}\label{HfinPhilemmaeq2}
 \int\limits_{G(\Z_p)}\rho_p(g)f\,dg
\end{equation}
also maps to $R^\ell\Phi_2$. We may thus assume that $f=\otimes_{p<\infty}f_{1/2,p}^{\rm sph}$.

Assume that $\phi$ has a non-trivial kernel $K$; we will obtain a contradiction. Considering $\H_{\rm fin}v$ for any non-zero vector $v\in K$, we see that $K$ contains an $\H_{\rm fin}$-invariant subspace of the form $W\otimes\big(\bigotimes_{p\nmid N}V_{1/2,p}\big)$, where $N>1$ is a square-free positive integer, and $W\subset\bigotimes_{p\mid N}V_{1/2,p}$ is invariant under $\bigotimes_{p\mid N}\H_p$. By Lemma~\ref{algebraactionlemma3}, $W$ contains $\bigotimes_{p\mid N}\mathcal{D}_p$. Hence
\begin{equation}\label{HfinPhilemmaeq3}
 \bigg(\bigotimes_{p\mid N}\mathcal{D}_p\bigg)\otimes\bigg(\bigotimes_{p\nmid N}V_{1/2,p}\bigg)\subset K.
\end{equation}
Let $B_N$ be as in \eqref{Eisleveleq11}. We have
\begin{equation}\label{HfinPhilemmaeq4}
 B_Nf=\Big(\otimes_{p\mid N}f_{1/2,p}^{\rm St}\Big)\otimes\Big(\otimes_{p\nmid N}f_{1/2,p}^{\rm sph}\Big)
\end{equation}
by definition of $B_N$. By \eqref{HfinPhilemmaeq3}, $B_Nf\in K$, so that $\phi(B_Nf)=0$. On the other hand, $\phi(B_Nf)=B_N\phi(f)=B_NR^\ell\Phi_2=R^\ell B_N\Phi_2$. By \eqref{Eisleveleq12}, $B_N\Phi_2$ is the automorphic form corresponding to $E_{2,N}$. Using Proposition~\ref{diffopHproposition} and \eqref{tildeENtheoremeq1}, it is easy to see that $R^\ell B_N\Phi_2\neq0$. This is the desired contradiction.
\end{proof}

For the next result, recall the definitions \eqref{Ukdefeq} of $U_k$ and \eqref{V12ppropeq6} of $V_{1/2}''$.

\begin{proposition}\label{V12ppEinjectiveprop}
 Consider the map $\mathbf{E}$ from $V_{1/2}''$ to the space of automorphic forms given by $f\mapsto E(\cdot,f)$. Let $\Phi_2$ be the image of $f_2$.
 \begin{enumerate}
  \item The image of $\mathbf{E}$ is contained in $\H\Phi_2$, and is $1$-codimensional in this space. An element of $\H\Phi_2$ which is not in the image is the constant function $1$.
  \item The map $\mathbf{E}$ is injective.
 \end{enumerate}
\end{proposition}
\begin{proof}
i) It follows from Proposition~\ref{diffopHproposition} that $L\Phi_2$ is a constant automorphic form. Hence, if we denote the space of constant automorphic forms simply by $\C$, then $\C\subset\H\Phi_2$.

By the PBW theorem, $\mathcal{U}(\mathfrak{g})\Phi_2$ has a basis $L\Phi_2,\Phi_2,R\Phi_2,R^2\Phi_2,\ldots$. Hence
\begin{equation}\label{V12ppEinjectivepropeq1}
 \H_\infty\Phi_2=\C\,\oplus\,\bigoplus_{\ell=0}^\infty\C(R^\ell\Phi_2)\,\oplus\,\bigoplus_{\ell=0}^\infty\C(\varepsilon_-R^\ell\Phi_2).
\end{equation}
It follows that
\begin{equation}\label{V12ppEinjectivepropeq2}
 \H\Phi_2=\C\,\oplus\,\bigoplus_{\ell=0}^\infty R^\ell\H_{\rm fin}\Phi_2\,\oplus\,\bigoplus_{\ell=0}^\infty\varepsilon_-R^\ell\H_{\rm fin}\Phi_2.
\end{equation}
For $k\geq2$, set $\ell=(k-2)/2$. Then, using Lemma~\ref{intpropprop},
\begin{equation}\label{V12ppEinjectivepropeq3}
 \mathbf{E}(U_k)=\mathbf{E}(R^\ell U_2)=\mathbf{E}(R^\ell\H_{\rm fin}f_2)=R^\ell\H_{\rm fin}\Phi_2.
\end{equation}
Similarly, for $k\leq-2$, we see $\mathbf{E}(U_k)=\varepsilon_-R^\ell\H_{\rm fin}\Phi_2$, where $\ell=(-k-2)/2$. In view of \eqref{V12ppEinjectivepropeq2}, this proves our assertions.

ii) By weight considerations, it is enough to prove that the restriction of $\mathbf{E}$ to $U_k$ is injective. By Lemma~\ref{intpropprop} i), this restriction is an $\H_{\rm fin}$-map. Clearly, $U_k\cong\bigotimes_{p<\infty}V_{1/2,p}$ as $\H_{\rm fin}$-modules. Hence the injectivity follows from Lemma~\ref{HfinPhilemma}.
\end{proof}

\begin{theorem}\label{E2reptheorem}
 Let $\Phi_2=E(\cdot,f_2)$ be the automorphic form corresponding to $E_2$. Then the global representation $\H\Phi_2$ generated by $\Phi_2$ admits a filtration
 \begin{equation}\label{E2reptheoremeq0}
  0\subset\C\subset\H\Phi_2,
 \end{equation}
 where $\C$ is the space of constant automorphic forms, and
 \begin{equation}\label{E2reptheoremeq1}
  \H\Phi_2/\C\cong\mathcal{D}_1^{\rm hol}\,\otimes\,\bigotimes_{p<\infty}V_{1/2,p}
 \end{equation}
 as $\H$-modules.
\end{theorem}
\begin{proof}
The map $\mathbf{E}$ from Proposition~\ref{V12ppEinjectiveprop} induces a vector space isomorphism $V_{1/2}''\to\H\Phi_2/\C$. Since $L(\mathbf{E}(f_2))=L\Phi_2$ is a constant automorphic form and $U_2=\H_{\rm fin}f_2$, it follows that $L(\mathbf{E}(f))\in\C$ for $f\in U_2$. Applying $\varepsilon_-$, it follows that $R(\mathbf{E}(f))\in\C$ for $f\in U_{-2}$. Combined with the intertwining properties of Lemma~\ref{intpropprop}, it follows that we constructed an isomorphism of $\H$-modules $V_{1/2}''\to\H\Phi_2/\C$. Now all we need to do is observe \eqref{V12ppropeq6}.
\end{proof}

Next we consider the global representations generated by the Eisenstein series of weight $2$ with level. For a positive, square-free integer $N$, we let
\begin{align}
 \label{fkNdefeq}f_{k,N}&=f_{1/2,\infty}^{(k)}\otimes\Big(\otimes_{p\mid N}f_{1/2,p}^{\rm St}\Big)\otimes\Big(\otimes_{p\nmid N}f_{1/2,p}^{\rm sph}\Big),\\
 \label{UkNdefeq}U_{k,N}&=\H_{\rm fin}f_{k,N}=\C f_{1/2,\infty}^{(k)}\otimes\Big(\bigotimes_{p\mid N}\mathcal{D}_p\Big)\otimes\Big(\bigotimes_{p\nmid N}V_{1/2,p}\Big),\\
 \label{V12ppNdefeq}V_{1/2,N}''&=\bigoplus_{\substack{k\in2\Z\\k\neq0}}U_{k,N}=\mathcal{D}_1^{\rm hol}\otimes\Big(\bigotimes_{p\mid N}\mathcal{D}_p\Big)\otimes\Big(\bigotimes_{p\nmid N}V_{1/2,p}\Big).
\end{align}
Observe that $\H f_{2,N}=V''_{1/2,N}$.

\begin{theorem}\label{E2Nreptheorem}
 For a square-free integer $N>1$ let $\Phi_{2,N}=E(\cdot,f_{2,N})$ be the automorphic form corresponding to $E_{2,N}$ (see Theorem~\ref{tildeENtheorem}). Then
 \begin{equation}\label{E2Nreptheoremeq1}
  \H\Phi_{2,N}\cong\mathcal{D}_1^{\rm hol}\otimes\Big(\bigotimes_{p\mid N}\mathcal{D}_p\Big)\otimes\Big(\bigotimes_{p\nmid N}V_{1/2,p}\Big).
 \end{equation}
\end{theorem}
\begin{proof}
We consider the restriction of the injective map $\mathbf{E}:V_{1/2}''\to\H\Phi_2$ to $V_{1/2,N}''$. Let $\gamma_p\in\H_p$ be the operator defined in \eqref{Eisleveleq7} and $B_N\in\H_{\rm fin}$ be as in \eqref{Eisleveleq11}. Since $L\Phi_2$ is a constant automorphic form, $\gamma_p(L\Phi_2)=0$. Hence $B_N(L\Phi_2)=0$, and it follows that
\begin{align}\label{E2Nreptheoremeq2}
 L(\mathbf{E}(U_{2,N}))&=L(\mathbf{E}(\H_{\rm fin}f_{2,N}))\nonumber\\
 &=\H_{\rm fin}L(\mathbf{E}(f_{2,N}))\nonumber\\
 &=\H_{\rm fin}L(\mathbf{E}(B_Nf_2))\nonumber\\
 &=\H_{\rm fin}B_NL(\mathbf{E}(f_2))\nonumber\\
 &=\H_{\rm fin}B_NL\Phi_2\nonumber\\
 &=0.
\end{align}
By applying $\varepsilon_-$ it follows that $R(\mathbf{E}(U_{-2,N}))=0$. Lemma~\ref{intpropprop} now implies that the map $\mathbf{E}:V_{1/2,N}''\to\H\Phi_2$ is $\H$-intertwining. Hence $V_{1/2,N}''=\H f_{2,N}\cong\H\mathbf{E}(f_{2,N})=\H\Phi_{2,N}$ as $\H$-modules.
\end{proof}
\section{Eisenstein series with character}\label{Sect:ES with character}
In the previous sections we have explained the adelic origin of the Eisenstein series $E_k$, including the case $k=2$. The literature also contains Eisenstein series $E_{k,\xi}$, whose definition involves a Dirichlet character $\xi$. In this section we apply the necessary modifications to our previous theory in order to account for the presence of a non-trivial $\xi$. The arguments will be slightly more complicated in some places, but easier in others, due to the fact that the global intertwining operator vanishes.

Instead of the $V_s$, the relevant induced representations will now be the $V_{s,\chi}:=\chi|\cdot|^s\times\chi^{-1}|\cdot|^{-s}$, where $\chi$ is a character of $F^\times$ in the context of a $p$-adic field $F$, or a character of $\R^\times$ in the archimedean case, or a character of $\Q^\times\backslash\A^\times$ in the global context. In the $p$-adic case, this representation is reducible if and only if $\chi^2=|\cdot|^{-2s\pm1}$, i.e., if $\chi=\mu|\cdot|^{-s\pm1/2}$ with a quadratic character $\mu$. In this case $V_{s,\chi}=\mu|\cdot|^{\pm1/2}\times\mu|\cdot|^{\mp1/2}=\mu V_{\pm1/2}$, and we have the exact sequences \eqref{padicexacteq1} and \eqref{padicexacteq2}.

\subsection{Dirichlet characters}
In the following we fix a primitive Dirichlet character $\xi$ of conductor $u>1$. Then there exists a unique character $\chi=\otimes_{v\leq\infty}\chi_v$ of $\Q^\times\backslash\A^\times$ with the property
\begin{equation}\label{Diri1}
 \prod_{p\mid u}\chi_p(a)=\xi(a)^{-1}\qquad\text{for }a\in\Z\text{ with }(a,u)=1.
\end{equation}
For a prime $p\nmid u$ the local component $\chi_p$ is unramified with Satake parameter $\chi_p(p)=\xi(p)$. For a prime $p\mid u$ the local component $\chi_p$ is ramified with conductor exponent $a(\chi_p)=v_p(u)$. The archimedean component is given by
\begin{equation}\label{Diri3}
 \chi_{\infty}=\begin{cases}
		{\rm triv} & \text{if } \xi(-1)=1,\\	
		\sgn & \text{if } \xi(-1)=-1.
	\end{cases}
\end{equation}
Since $\chi(a)=1$, it follows from \eqref{Diri1} that
\begin{equation}\label{Diri2}
 \prod_{p\mid a}\chi_p(a)=\xi(|a|)\qquad\text{for }a\in\Z\text{ with }(a,u)=1.
\end{equation}
The classical Gauss sum of $\xi$ is defined by
\begin{equation}\label{Gausssumeq}
 G(\xi)=\sum_{a=1}^u\xi(a)e^{\frac{2\pi ia}{u}}.
\end{equation}
As in Sect.~\ref{globalWsec}, we fix the global additive character $\psi=\prod\psi_v$ as in Tate's thesis. Attached to $\chi$ is a global $\varepsilon$-factor
\begin{equation}\label{globalepsilondefeq}
 \varepsilon(s,\chi)=\prod_{v\leq\infty}\varepsilon(s,\chi_v,\psi_v),
\end{equation}
defined as a product of local factors. The archimedean $\varepsilon$-factor is independent of the complex variable $s$, and is given by
\begin{equation}\label{archepsiloneq}
 \varepsilon(s,\chi_\infty,\psi_\infty)=
 \begin{cases}
  1&\text{if }\chi_\infty={\rm triv},\\
  -i&\text{if }\chi_\infty=\sgn.
 \end{cases}
\end{equation}
\begin{lemma}\label{epsgausslemma}
 With the above notations and conventions,
 \begin{equation}\label{epsgausslemmaeq}
  \prod_{p<\infty}\varepsilon(0, \chi_p,\psi_p)=G(\xi).
 \end{equation}
\end{lemma}
\begin{proof}
This is an exercise, using the standard formula
\begin{equation}\label{epsgausslemmaeq2}
 \varepsilon(0, \chi_p,\psi_p)=\int\limits_{p^{-a(\chi_p)}\Z_p^\times}\chi_p^{-1}(x)\psi_p(x)\,dx
\end{equation}
as a starting point.
\end{proof}

Since $\xi$ is primitive, we have an equality of $L$-functions $L(s,\xi)=L(s,\chi)$, where $L(s,\xi)=\sum_{n=1}^\infty \xi(n)n^{-s}$ is the classical Dirichlet $L$-series. In the following the Dirichlet character $\xi^2$ will be relevant, by which we mean the function $\xi^2(n)=\xi(n)^2$ for $n\in\Z$ (as opposed to the primitive Dirichlet character associated to this function). We have $L(s,\xi^2)=\prod_{p\nmid u}L(s,\chi_p^2)$, but not in general $L(s,\xi^2)=L(s,\chi^2)$.
\subsection{Intertwining operators}\label{indrepsec-char}
\subsubsection*{Non-archimedean case}
Let $F$, $\varpi$, $q$, $|\cdot|$, $v$, $G$ be as in Sect.~\ref{indrepsec}. Let $\chi$ be a character of $F^\times$. Recall that $V_{s,\chi}$ is the standard space of the parabolically induced representation $\chi|\cdot|^s\times \chi^{-1}|\cdot |^{-s}$, consisting of locally constant functions $f:G\to\C$ with the transformation property
\begin{equation}\label{localpadiceq-1}
	f\left(\left[\begin{smallmatrix}a&b\\&d \end{smallmatrix}\right]g\right)=\Big|\frac{a}{d}\Big|^{s+1/2} \chi\left(\frac{a}{d}\right)f(g).
\end{equation}
Set
\begin{equation}\label{betadefeq}
 \beta:=\begin{cases}
         \chi^2(\varpi)&\text{if $\chi^2$ is unramified},\\
         0&\text{if $\chi^2$ is ramified}.
        \end{cases}
\end{equation}
Generalizing \eqref{localpadiceq1}, we define for $f_s\in V_{s,\chi}$
\begin{equation}\label{localpadiceq88}
	(A_{s,\chi}f_s)(g):=\lim_{N\to\infty}\bigg(\int\limits_{\substack{F\\v(b)>-N}}f_s(\begin{bsmallmatrix}&-1\\1\end{bsmallmatrix}\begin{bsmallmatrix}1&b\\&1\end{bsmallmatrix}g)\,db+(1-q^{-1})\frac{\beta^N q^{-2Ns}}{1-\beta q^{-2s}} f_s(g)\bigg),
\end{equation}
assuming $q^{2s}\neq\beta$. By a similar argument as in Sect.~\ref{interpadic}, we see that $A_{s,\chi}$ is an intertwining operator $V_{s,\chi}\to V_{-s,\chi^{-1}}$, for any $s\in\C$ with $q^{2s}\neq\beta$.

Assume that $\chi$ itself is unramified. Then we have a normalized spherical vector $f_{s,\chi}^{\rm sph}$. We also define the \emph{Steinberg vector} $f_{s,\chi}^{\rm St}$ by
\begin{equation}\label{localpadiceq4 char}
 f_{s,\chi}^{\rm St}=\frac1{1-\beta^{-1} q^{2s+1}}\Big((1+\beta^{-1} q^{2s})f_{s,\chi}^{\rm sph}-\chi(\varpi)^{-1}q^{s-1/2}(q+1)\mat{1}{}{}{\varpi}f_{s,\chi}^{\rm sph}\Big),
\end{equation}
which satisfies $f_{s,\chi}^{\rm St}(1)=1$ and $f_{s,\chi}^{\rm St}(\mat{}{-1}{1}{})=-q^{-1}$. The two vectors $f_{s,\chi}^{\rm sph}$ and $f_{s,\chi}^{\rm St}$ form a basis of $V_{s,\chi}^{\Gamma_0(\p)}$. 
Calculations show that
\begin{align}
 \label{localpadiceq5char}A_{s,\chi}f^{\rm sph}_{s,\chi}&=\frac{1-\beta q^{-2s-1}}{1-\beta q^{-2s}}\,f^{\rm sph}_{-s,\chi^{-1}},\\
 \label{localpadiceq6char}A_{s,\chi}f_{s,\chi}^{\rm St}&=-q^{-1}\frac{1-\beta q^{-2s+1}}{1-\beta q^{-2s}}f_{-s,\chi^{-1}}^{\rm St}.
\end{align}
In particular, if $\chi=\mu|\cdot|^{-s+1/2}$ with an unramified quadratic character $\mu$, then $V_{s,\chi}=V_{1/2,\mu}$, and
\begin{equation}\label{localpadiceq7 char}
 f_{1/2,\mu}^{\rm St}=\frac1{1-q}\Big(f_{1/2,\mu}^{\rm sph}-\mu(\varpi)\mat{1}{}{}{\varpi}f_{1/2,\mu}^{\rm sph}\Big)
\end{equation}
lies in the kernel of $A_{1/2,\mu}$.
\subsubsection*{Archimedean case}
Let $\chi$ be either the trivial character or the sign character of $\R^\times$. Then $V_{s,\chi}$ is the $\H_\infty$-module spanned by the functions $f_{s, \chi}^{(k)}$ for $k\in2\Z$, where
\begin{equation}\label{localrealeq2-char0}
 f_{s, \chi}^{(k)}(\mat{a}{b}{}{d}r(\theta))=\chi(ad^{-1})\Big|\frac ad\Big|^{s+1/2}e^{ik\theta},
\end{equation}
for $a,d\in\R^\times$, $b\in\R$, and $r(\theta)$ as in \eqref{rthetadefeq}. Lemma~\ref{RHLVslemma} still holds with $f_{s, \chi}^{(k)}$ instead of $f_s^{(k)}$, except that \eqref{RHLVslemmaeq4} has to be replaced by $\varepsilon_-f_{s, \chi}^{(k)}=\chi(-1)f_{s, \chi}^{(-k)}$. We still have the sequences \eqref{realexacteq1} and \eqref{realexacteq2}, with $V_{s,\chi}$ instead of $V_s$ and $\chi\mathcal{F}_{k-1}$ instead of $\mathcal{F}_{k-1}$; there is no need to twist $\mathcal{D}_{k-1}^{\rm hol}$, since it is invariant under twisting by the sign character.

For the archimedean intertwining operator $A_{s,\chi}:V_{s,\chi}\to V_{-s,\chi}$, defined just as in \eqref{localrealeq3} and convergent for ${\rm Re}(s)>0$, we still have the statement of Lemma~\ref{localreallemma1}.
\subsubsection*{Global case}
Now let $\chi= \otimes_{v\le \infty} \chi_v$ be the character of $\Q^\times\backslash\A^{\times}$ corresponding to our primitive Dirichlet character $\xi$, and consider the global $\H$-module
\begin{equation}\label{globalVschieq}
 V_{s,\chi}=\bigotimes_v V_{s,\chi_v}
\end{equation}
with the local $\H_v$-modules $V_{s,\chi_v}$ defined above. We would like to define a global intertwining operator $A_{s,\chi}:V_{s,\chi}\to V_{-s,\chi^{-1}}$ by an integral as in \eqref{globaleq2}. To investigate convergence, let $f=\otimes f_v$ be a pure tensor in $V_{s,\chi}$. Let $T$ be a finite set of finite places such that $\chi_p$ is unramified and $f_p=f_{s,\chi_p}^{\rm sph}$ for $p\notin T$; such a set $T$ exists by definition of the restricted tensor product. Similar to \eqref{globaleq4}, we get
\begin{equation}\label{globaleq-4}
	(A_{s,\chi}f)(g)=\bigg(\prod_{v\in T\cup\{\infty\}}\:(A_{s,\chi_v}f_v)(g_v)\bigg)\bigg(\prod_{p\notin T}\frac{1-\chi_p^2(p) p^{-2s-1}}{1-\chi_p^2(p)  p^{-2s}}\,f^{\rm sph}_{-s,\chi_p^{-1}}(g_p)\bigg).
\end{equation}
Again we see that the product converges for ${\rm Re}(s)>1/2$. We may rewrite \eqref{globaleq-4} as
\begin{equation}\label{globaleq-5}
	(A_{s,\chi}f)(g)=\frac{L(2s,\chi^2)}{L(2s+1,\chi^2)}((A_{s,\chi_\infty}f_\infty)(g_\infty))\bigg(\prod_{p\in T}\:\frac{1-\chi_p^2(p)p^{-2s}}{1-\chi_p^2(p)p^{-2s-1}}(A_{s,\chi_p}f_p)(g_p)\bigg)\!\big(\prod_{p\notin T}f^{\rm sph}_{-s,\chi_p^{-1}}(g_p)\big).
\end{equation}
Now assume that $f=f_s=\otimes f_{s,v}$ varies in a flat section. Assume further that $f_{s,\infty}=f_{s,\chi_\infty}^{(k)}$ for some even $k\geq2$. Then, by \eqref{localrealeq6}, the limit $\lim_{s\to1/2}(A_{s,\chi_\infty}f_{s, \infty})(g_\infty)$ is zero. It follows that the right hand side of \eqref{globaleq-5} is analytic in the region ${\rm Re}(s)>0$. In fact:
\begin{itemize}
 \item If $\chi^2\neq1$, then $L(2s,\chi^2)$ is holomorphic in ${\rm Re}(s)>0$, so that $(A_{s,\chi}f)(g)$ vanishes at $s=1/2$.
 \item If $\chi^2=1$, then $L(2s,\chi^2)$ has a simple pole at $s=1/2$. However, in this case we will let $f_{s,p}$ be the flat section containing an element of the subrepresentation $\chi_p\St_{\GL(2,\Q_p)}$ of $V_{1/2,\chi_p}$. Note that by our assumptions on $\xi$ the set $T$ is non-empty, and that $f_{1/2,\chi_p}$ is in the kernel of $A_{1/2,\chi_p}$. Hence we still obtain that $(A_{s,\chi}f)(g)$ vanishes at $s=1/2$.
\end{itemize}
\subsection{Whittaker integrals}
\subsubsection*{Non-archimedean case}
We use the same $p$-adic setup as in the previous section. For $f\in V_{s,\chi}$ and $\alpha\in F^\times$ we define $W_{s,\chi}^\alpha$ by the same formula as in \eqref{localpadicWeq2}. Instead of \eqref{localpadicWeq10}, we have
\begin{equation}\label{localpadicWeq10b}
 (W_{s,\chi}^\alpha f)(\mat{y}{}{}{1}g)=\chi(y)^{-1}|y|^{1/2-s}\,(W_{s,\chi}^{\alpha y}f)(g)
\end{equation}
for all $g\in G$ and $y,\alpha\in F^\times$.

Assume that $\chi$ is ramified with conductor exponent $a(\chi)$. Then $V_{s,\chi}^{\Gamma_0(\p^n)}=0$ for $0\leq n<2a(\chi)$, and $\dim V_{s,\chi}^{\Gamma_0(\p^n)}=1$ for $n=2a(\chi)$. A non-zero $\Gamma_0(\p^{2a(\chi)})$-invariant vector is given in \cite[Prop.~2.1.2]{Schmidt2002}. It is supported on the double coset $B\mat{1}{}{\varpi^{a(\chi)}}{1}\Gamma_0(\p^{2a(\chi)})$, where $B$ is the upper triangular subgroup of $\GL(2,F)$. We call such a vector $f_{s,\chi}^{\rm new}$, and normalize it so that
\begin{equation}\label{fnewdefeq}
 f_{s,\chi}^{\rm new}(\mat{1}{}{\varpi^{a(\chi)}}{1})=\chi(\varpi)^{-a(\chi)}.
\end{equation}
Note that the definition is independent of the choice of uniformizer. If $\chi^2\neq|\cdot|^{-2s\pm1}$, then $f_{s,\chi}^{\rm new}$ is the newform in the irreducible principal series representation $V_{s,\chi}$. If $\chi=\mu|\cdot|^{-s+1/2}$ with a quadratic character $\mu$, then $f_{s,\chi}^{\rm new}$ is the newform in the subrepresentation $\mu\St_{\GL(2)}$ of $V_{s,\chi}$.

\begin{lemma}\label{calphaplemma1}
 The following holds for any $\alpha, y\in F^{\times}$.
 \begin{enumerate}
  \item Suppose that $\chi$ is unramified. Let $\beta=\chi^2(\varpi)$ and assume that $q^{2s}\neq\beta$. If $f_{s,\chi}^{\rm sph}\in V_{s,\chi}$ is the normalized spherical vector, then
   \begin{align}\label{calphaplemma-eq3}
    &(W^\alpha_{s,\chi} f_{s,\chi}^{\rm sph})(\mat{y}{}{}{1})\nonumber\\
     &=\begin{cases}
      \displaystyle\frac{(\chi^2(\alpha)|\alpha|^{2s}\chi(y)|y|^{1/2+s}-\beta^{-1} q^{2s}\chi(y)^{-1}|y|^{1/2-s})(1-\beta q^{-2s-1})}{1-\beta^{-1} q^{2s}}&\text{if }v(\alpha y)\geq0,\\
      0&\text{if }v(\alpha y)<0.
      \end{cases}
   \end{align}
  \item Suppose that $\chi$ is unramified. Let $\beta=\chi^2(\varpi)$ and assume that $q^{2s}\neq\beta$. If $f_{s,\chi}^{\rm St}\in V_{s,\chi}$ is the Steinberg vector defined in \eqref{localpadiceq4 char}, then
   \begin{align}\label{calphaplemma-eq31}
    &(W^\alpha_{s,\chi} f_{s,\chi}^{\rm St})(\mat{y}{}{}{1})\nonumber\\
    &=\begin{cases}
       \displaystyle\frac{(1-\beta q^{-2s-1})\chi^2(\alpha)|\alpha|^{2s}\chi(y)|y|^{1/2+s}-(1-\beta^{-1}q^{2s-1})\chi(y)^{-1}|y|^{1/2-s}}{1-\beta^{-1}q^{2s}}&\text{if }v(\alpha y)\geq0,\\
       0&\text{if }v(\alpha y)<0.
      \end{cases}
   \end{align}
  \item Suppose $a(\chi)>0$. Then, with $f_{s,\chi}^{\rm new}$ as in \eqref{fnewdefeq},
   \begin{equation}\label{calphaplemma-eq4}
    (W_{s,\chi}^\alpha f_{s,\chi}^{\rm new})(\mat{y}{}{}{1})=
    \begin{cases}
     0 &\text{if } v(\alpha y)\neq 0,\\
     \displaystyle q^{-a(\chi)(2s+1)}|y|^{1/2-s} \chi(-\alpha)\varepsilon(0,\chi,\psi)&\text{if } v(\alpha y)=0.
    \end{cases} 
  \end{equation}
 \end{enumerate}
\end{lemma}
\begin{proof}
i) and ii) follow by setting $\chi=|\cdot|^t$ and replacing $s$ by $s+t$ in Lemma~\ref{calphaplemma}. For iii) see \cite[Lemma~2.2.1]{Schmidt2002}.
\end{proof}
\subsubsection*{Archimedean case}
We consider the archimedean $V_{s,\chi}$, where $\chi$ is either trivial or the sign character. Recall that it is spanned by the functions in \eqref{localrealeq2-char0}. Generalizing \eqref{localrealWeq5}, we have
\begin{equation}\label{localrealWeq5-char}
	(W_{(k-1)/2,\chi}^\alpha f_{(k-1)/2, \chi}^{(k)})\left(\mat{1}{x}{}{1}\mat{y}{}{}{1}\right)=
	\begin{cases}
		0&\text{if }\alpha y>0,\\[1ex]
		\displaystyle\frac{(2\pi i)^k}{(k-1)!}\, \chi(y) |y|^{k/2}|\alpha|^{k-1}e^{-2\pi i\alpha(x+iy)}&\text{if }\alpha y <0.
	\end{cases}
\end{equation}
In our application we will have $y>0$, in which case there is no difference to the previous formula.
\subsubsection*{Global case}
Let $\xi$ be our primitive Dirichlet character of conductor $u>1$, and $\chi$ the corresponding character of $\Q^\times\backslash\A^\times$. We consider the global Whittaker functional $W_{s,\chi}^\alpha$ defined as in \eqref{globalWeq1} with $f\in V_{s,\chi}$, with the same global additive character~$\psi=\prod\psi_v$. Then Lemma~\ref{alphalatticelemma} is still true. 

For a positive integer $n$, we define the generalized power sum as follows,
\begin{equation}\label{sigma-char}
	\sigma_t^{\xi}(n)=\xi(n)\sum_{m\mid n}\xi^{-2}(m)m^t,
\end{equation}
the sum being taken over positive divisors of $n$. We write $\sigma^\xi$ for $\sigma^\xi_1$. The following lemma is analogous to Lemma~\ref{Walphasigmalemma}.
\begin{lemma}\label{Walphasigmalemma-char}
 Let $\alpha$ be a non-zero integer. If $(\alpha,u)=1$, then the following formulas hold.
 \begin{enumerate}
  \item For ${\rm Re}(s) > 0$,
   \begin{align}\label{Walphasigmalemma-eq21}
    &\bigg(\prod_{p \mid u}(W^\alpha_{s,\chi_p} f_{s,\chi_p}^{\rm new})(1)\bigg)\bigg(\prod_{p\nmid u}(W^\alpha_{s,\chi_p} f_{s,\chi_p}^{\rm sph})(1)\bigg)\nonumber\\
    &\hspace{10ex}=\frac{\chi_\infty(-\alpha)|\alpha|^{-2s}}{u^{2s+1}L(2s+1,\xi^2)}\bigg(\prod_{p<\infty}\varepsilon(0,\chi_p,\psi_p)\bigg)\sigma^\xi_{2s}(|\alpha|).
   \end{align}
  \item For $s = 1/2$ and a positive, square-free integer $N$ with $(N,u)=1$ and  $\chi_p^2=1$ for $p\mid N$,
   \begin{align}\label{Walphasigmalemma-eq22}
    &\bigg(\prod_{p \mid u}(W^\alpha_{1/2,\chi_p} f_{1/2,\chi_p}^{\rm new})(1)\bigg)\bigg(\prod_{p\mid N}(W^\alpha_{1/2,\chi_p} f_{1/2,\chi_p}^{\rm St})(1)\bigg)\bigg(\prod_{p\nmid uN}(W^\alpha_{1/2,\chi_p} f_{1/2,\chi_p}^{\rm sph})(1)\bigg)\nonumber\\
    &\hspace{10ex}=\frac{\chi_\infty(-\alpha)|\alpha|^{-1}\mu(N)}{u^2L(2,\xi^2)\varphi(N)}\bigg(\prod_{p<\infty}\varepsilon(0,\chi_p,\psi_p)\bigg)\xi(|\alpha/\alpha'|)\,\sigma^\xi(|\alpha'|),
    \end{align}
   where $\alpha'$ is the part of $\alpha$ relatively prime to $N$.
 \end{enumerate}
 If $(\alpha,u)\neq1$, then the left hand sides of \eqref{Walphasigmalemma-eq21} and \eqref{Walphasigmalemma-eq22} are zero.
\end{lemma}
\begin{proof}
The last statement follows from \eqref{calphaplemma-eq4}. We will therefore assume that $(\alpha,u)=1$.

i) From \eqref{calphaplemma-eq3} and \eqref{calphaplemma-eq4} we get
\begin{align*}
 &\bigg(\prod_{p \mid u}(W^\alpha_{s,\chi_p} f_{s,\chi_p}^{\rm new})(1)\bigg)\bigg(\prod_{p\nmid u}(W^\alpha_{s,\chi_p} f_{s,\chi_p}^{\rm sph})(1)\bigg)\\
 &=\bigg(\prod_{p \mid u}p^{-a(\chi_p)(2s+1)}
  \chi_p(-\alpha)\varepsilon(0,\chi_p,\psi_p)\bigg)\bigg(\prod_{p\nmid u}\frac{(\chi_p^2(\alpha)|\alpha|_p^{2s}-\beta_p^{-1} p^{2s})(1-\beta_p p^{-2s-1})}{1-\beta_p^{-1} p^{2s}}\bigg)\\
 &=\frac1{\prod_{p\nmid u}L(2s+1,\chi_p^2)}\bigg(\prod_{\substack{p\mid u}}p^{-v_p(u)(2s+1)}\chi_p(-\alpha)\varepsilon(0,\chi_p,\psi_p) \bigg)\bigg(\prod_{\substack{p\nmid u}}\frac{\chi_p^2(\alpha)|\alpha|_p^{2s}-\beta_p^{-1} p^{2s}}{1-\beta_p^{-1} p^{2s}}\bigg)\\
 &=\frac{1}{u^{2s+1}L(2s+1,\xi^2)}\bigg(\prod_{\substack{p\mid u}}\chi_p(-\alpha)\varepsilon(0,\chi_p,\psi_p) \bigg)\bigg(\prod_{p\nmid u}\chi_p^2(\alpha)|\alpha|_p^{2s}\bigg)\bigg(\prod_{p\nmid u}\frac{1-(\beta_p^{-1} p^{2s})^{v_p(\alpha)+1}}{1-\beta_p^{-1} p^{2s}}\bigg)\\
 &=\frac{\chi_\infty(-\alpha)|\alpha|_\infty^{-2s}}{u^{2s+1}L(2s+1,\xi^2)}\bigg(\prod_{p\mid u}\varepsilon(0,\chi_p,\psi_p) \bigg)\bigg(\prod_{p\nmid u}\chi_p(-\alpha)\bigg)\bigg(\prod_{p\nmid u}(1+\beta_p^{-1}p^{2s}+\cdots+(\beta_p^{-1}p^{2s})^{v_p(\alpha)})\bigg)\\
 &\stackrel{\eqref{Diri2}}{=}\frac{\chi_\infty(-\alpha)|\alpha|_\infty^{-2s}\xi(|\alpha|)}{u^{2s+1}L(2s+1,\xi^2)}\bigg(\prod_{p<\infty}\varepsilon(0,\chi_p,\psi_p) \bigg)\bigg(\prod_{p\nmid u}(1+\xi(p)^{-2}p^{2s}+\cdots+(\xi(p)^{-2}p^{2s})^{v_p(\alpha)})\bigg)\\
 &=\frac{\chi_\infty(-\alpha)|\alpha|_\infty^{-2s}\xi(|\alpha|)}{u^{2s+1}L(2s+1,\xi^2)}\bigg(\prod_{p<\infty}\varepsilon(0,\chi_p,\psi_p) \bigg)\bigg(\sum_{m\mid\alpha}\xi(m)^{-2}m^{2s}\bigg).
\end{align*}
In view of \eqref{sigma-char}, this proves i).

ii) follows from a similar calculation, using \eqref{calphaplemma-eq3}, \eqref{calphaplemma-eq31} and \eqref{calphaplemma-eq4}.
\end{proof}
\subsection{The classical Eisenstein series with character}\label{Eischarsec}
We continue to let $\xi$ be our  non-trivial primitive Dirichlet character of conductor $u>1$. For an even integer $k\geq2$, let
\begin{equation}\label{Eisexampleseq6-char}
	E_{k,\xi}(z)=C(k,\xi)\sum_{n=1}^\infty\sigma_{k-1}^{\xi}(n)e^{2\pi inz},
\end{equation}
with $\sigma_{k-1}^\xi$ defined in \eqref{sigma-char}, and
\begin{equation}\label{den of C(k, xi)}
    C(k,\xi)=\displaystyle\frac{(2\pi i)^k\,G(\xi)}{(k-1)!\,u^k L(k,\xi^{2})}.
\end{equation}
We note that this normalization of the Eisenstein series differs from the one in \cite[Sects.~4.5, 4.6]{DiamondShurman2005}. We choose the form in \eqref{Eisexampleseq6-char} since the following result works out nicely coming from the adelic Eisenstein series.

\begin{theorem}\label{Ekchartheorem}
 Let $\xi$ be a primitive Dirichlet character modulo $u>1$, and let $\chi$ be the corresponding character of $\Q^\times\backslash\A^\times$. Let $k\geq2$ be an even integer and $f_{k,\chi}\in V_{(k-1)/2,\chi}$ be the following pure tensor,
	\begin{equation}\label{Ekchartheoremeq0}
		f_{k,\chi}=f_{(k-1)/2,\chi_\infty}^{(k)}\otimes\Big(\otimes_{p\mid u}f_{(k-1)/2,\chi_p}^{\rm new}\Big)\otimes\Big(\otimes_{p\nmid u}f_{(k-1)/2,\chi_p}^{\rm sph}\Big).
	\end{equation}
	Then $E(\cdot,f_{k,\chi})$ is the automorphic form corresponding to $E_{k,\xi}$. It is a holomorphic modular form of weight $k$ with respect to $\Gamma_0(u^2)$.
\end{theorem}
\begin{proof}
If $k\geq4$, then we have, similar to Lemma~\ref{Fourierlemma1},
\begin{equation}\label{Ekchareq2}
	E(g,f_{k,\chi})=f_{k,\chi}(g)+(A_{(k-1)/2,\chi}f_{k,\chi})(g)+\sum_{\substack{\alpha\in\Z\\\alpha\neq0}}(W_{(k-1)/2,\chi}^\alpha f_{k,\chi})(g).
\end{equation}
It follows from \eqref{globaleq-5} that $A_{(k-1)/2,\chi}f_{k,\chi}=0$. Since there is at least one prime $p\mid u$, it follows from the definition of $f^{\rm new}_{s,\chi_p}$ that $f_{k,\chi}(g)=0$ for archimedean $g$. Hence
\begin{equation}\label{Ekchareq2b}
	E(g,f_{k,\chi})=\sum_{\substack{\alpha\in\Z\\\alpha\neq0}}(W_{(k-1)/2,\chi}^\alpha f_{k,\chi})(g)\qquad \text{for }g\in G(\R).
\end{equation}
If $k=2$, we obtain the same identity by analytic continuation; see the comments after \eqref{globaleq-5}.

Applying \eqref{Fouriereq10} to our function \eqref{Ekchareq2b}, we see from Lemma~\ref{Walphasigmalemma-char} i) and \eqref{localrealWeq5-char} that the corresponding function on the upper half plane is given by
\begin{align}\label{Ekchareq3}
	F(z)&=y^{-k/2}\sum_{\substack{\alpha\in\Z\\\alpha\neq0}}(W_{(k-1)/2, \chi}^\alpha f_{k,\chi})(\mat{1}{x}{}{1}\mat{y}{}{}{1})\nonumber\\
	&=y^{-k/2}\sum_{\substack{\alpha\in\Z_{<0}\\(\alpha,u)=1}}	\displaystyle\frac{(2\pi i)^k}{(k-1)!}|y|^{k/2}|\alpha|^{k-1}e^{-2\pi i\alpha(x+iy)}
	\frac{\chi_\infty(-\alpha)|\alpha|^{-(k-1)}}{u^k L(k,\xi^2)}\bigg(\prod_{p<\infty}\varepsilon(0,\chi_p,\psi_p)\bigg)\sigma^\xi_{k-1}(|\alpha|)\nonumber\\
	&=\displaystyle\frac{(2\pi i)^k}{(k-1)!}
	\frac{1}{u^k L(k,\xi^2)}\bigg(\prod_{p<\infty}\varepsilon(0,\chi_p,\psi_p)\bigg)\sum_{n=1}^{\infty} \sigma_{k-1}^{\xi}(n) e^{2\pi i n z}.
\end{align}
Now we use Lemma~\ref{epsgausslemma} to complete the proof.
\end{proof}

\begin{remark}
The constant term in \eqref{Eisexampleseq6-char} is zero, i.e., the Eisenstein series $E_{k,\xi}$ vanishes at the cusp at infinity. This stems from the fact that both $f_{k,\chi}(g)$ and $(A_{(k-1)/2,\chi}f_{k,\chi})(g)$ in \eqref{Ekchareq2} are zero for archimedean $g$. It follows from \eqref{globaleq-5} that in fact $(A_{(k-1)/2,\chi}f_{k,\chi})(g)=0$ for any $g\in G(\A)$. Hence, for any $h\in\SL(2,\Q)$,
\begin{align}\label{Ekchareq4}
 (E_{k,\xi}|h)(z)&=y^{-k/2}E(\mat{1}{x}{}{1}\mat{y}{}{}{1}h_{\rm fin}^{-1},f_{k,\chi})\nonumber\\
  &=f_{k,\chi}(h_{\rm fin}^{-1})+y^{-k/2}\sum_{\substack{\alpha\in\Z\\\alpha\neq0}}(W_{(k-1)/2, \chi}^\alpha f_{k,\chi})(\mat{1}{x}{}{1}\mat{y}{}{}{1}h_{\rm fin}^{-1}).
\end{align}
The cusps are in bijection with $B(\Q)\backslash\SL(2,\Q)/\Gamma_0(u^2)$. Looking at the support of the local newforms in \eqref{fnewdefeq}, we see from \eqref{Ekchareq4} that $E_{k,\xi}$ vanishes at all cusps except the one represented by $\mat{1}{}{-u}{1}$.
\end{remark}
\subsection{Eisenstein series of weight \texorpdfstring{$2$}{} with level \texorpdfstring{$u^2N$}{}}
We continue to assume that $\xi$ is a primitive Dirichlet character modulo $u>1$ and $\chi$ is the corresponding character of $\Q^\times\backslash\A^\times$. Let $N$ be a square-free positive integer with $(u,N)=1$ and such that $\chi_p^2=1$ for $p\mid N$. (Note that, for $p\nmid u$, we have $\chi_p^2=1$ if and only if $\xi(p)^2=1$; see \eqref{Diri2}.) Consider the element of $V_{1/2,\chi}$ defined by
\begin{equation}\label{Eisleveleq6-char}
	f_{2,N,\chi}:=f_{1/2,\chi_\infty}^{(2)}\otimes\Big(\otimes_{p\mid u}f_{1/2,\chi_p}^{\rm new}\Big)\otimes\Big(\otimes_{p\mid N}f_{1/2,\chi_p}^{\rm St}\Big)\otimes\Big(\otimes_{p\nmid uN}f_{1/2,\chi_p}^{\rm sph}\Big).
\end{equation}
See \eqref{localpadiceq4 char} for the definition of $f_{1/2,\chi_p}^{\rm St}$. In the following result $\mu$ is the M\"obius function and $\varphi$ is Euler's function.
\begin{theorem}\label{tildeENtheorem-char}
 Let $\xi,\chi,u,N$ be as above and $f_{2,N,\chi}\in V_{1/2,\chi}$ be as in \eqref{Eisleveleq6-char}. Let $C(2,\xi)$ be as in \eqref{den of C(k, xi)} for $k=2$. Then the function on the upper half plane corresponding to $E(\cdot,f_{2,N,\chi})$ is given by 
    \begin{equation}
	  E_{2,N,\xi}(z)=C(2, \xi)\frac{\mu(N)}{\varphi(N)}\sum_{n=1}^{\infty}\xi\left(\frac n{n'}\right)\sigma^{\xi}(n^\prime)e^{2\pi i n z},
	\end{equation}
 where $n'$ is the part of $n$ relatively prime to $N$. It is a holomorphic modular form of weight $2$ with respect to $\Gamma_0(u^2N)$.
\end{theorem}
\begin{proof}
The proof is similar to that of Theorem~\ref{Ekchartheorem}, making use of Lemma~\ref{Walphasigmalemma-char} ii) and \eqref{localrealWeq5-char}.
\end{proof}

Up to normalization, the following Eisenstein series is defined in \cite[Sect.~4.6] {DiamondShurman2005},
\begin{equation}\label{tildeENtheoremeq1-char classical DS}
	\tilde E_{2,N,\xi}(z)=N\cdot C(2, \xi)\sum_{n=1}^\infty\sigma^{\xi}(n)e^{2\pi inNz}=N\cdot E_{2,\xi}(Nz).
\end{equation}
It is a holomorphic modular form of weight $2$ with respect to $\Gamma_0(u^2N)$. It is easily verified that $E(\cdot,\tilde f_{2,N,\chi})$, where $\tilde f_{2,N,\chi}=\mat{1}{}{}{N_{\rm fin}}f_{2,\chi}$, is the automorphic form corresponding to $\tilde E_{2,N,\xi}$. Here $N_{\rm fin}$ is defined in \eqref{Eisleveleq4} and $f_{2,\chi}\in V_{1/2,\chi}$ is the function from Theorem~\ref{Ekchartheorem}.

There is a result analogous to Proposition~\ref{tildeENrelationprop} which relates $\tilde E_{2,N,\xi}$ to $E_{2,N,\xi}$. To derive it, we define the following elements of the local Hecke algebra $\H_p$,
\begin{equation}\label{Eisleveleq7 char}
 \alpha_p={\rm char}(K_p),\quad\beta_p={\rm char}(\mat{1}{}{}{p}K_p),\quad\gamma_p=\frac1{1-p}(\alpha_p-\chi_p(p)\beta_p),\quad \delta_p=\frac{\mathrm{char}(\Gamma_0(p^{v_p(u)}))}{\mathrm{vol}(\Gamma_0(p^{v_p(u)}))}.
\end{equation}
Here we assume $p\nmid u$ for $\beta_p,\gamma_p$, and $p\mid u$ for $\delta_p$. For a square-free, positive integer $N$, let $B_{N, \chi},\, \tilde B_{N, \chi}\in\H_{\rm fin}$ be defined by
\begin{align}
 \label{Eisleveleq10 char}\tilde B_{N,\chi}&=\bigg(\bigotimes_{p\mid u}\delta_p\bigg)\otimes\bigg(\bigotimes_{p\mid N}\beta_p\bigg)\otimes\bigg(\bigotimes_{p\nmid uN}\alpha_p\bigg),\\
 \label{Eisleveleq11 char}B_{N,\chi}&=\bigg(\bigotimes_{p\mid u}\delta_p\bigg)\otimes\bigg(\bigotimes_{p\mid N}\gamma_p\bigg)\otimes\bigg(\bigotimes_{p\nmid uN}\alpha_p\bigg).
\end{align}
It follows that
\begin{equation}\label{Eisleveleq12 char}
 \tilde B_{N,\chi}f_{2,\chi}=\tilde f_{2,N, \chi},\qquad B_{N,\chi}f_{2,\chi}=f_{2,N, \chi},
\end{equation}
where as before $f_{2,\chi}\in V_{1/2,\chi}$ is the function from Theorem~\ref{Ekchartheorem}. For the second equality, note that $\gamma_p(f_{1/2,\chi_p}^{\rm sph})=f_{1/2,\chi_p}^{\rm St}$ by \eqref{localpadiceq7 char}.
\begin{proposition}\label{tildeENrelationprop char}
 For all square-free, positive integers $N$,
 \begin{equation}\label{tildeENrelationpropeq1 char}
  E_{2,N, \xi}=\frac{1}{\varphi(N)}\sum_{M\mid N}\xi(M)\mu(N/M)\tilde E_{2,M,\xi},\qquad \tilde E_{2,N,  \xi}=\xi(N)\sum_{M\mid N}\varphi(M)E_{2,M, \xi}.
 \end{equation}
\end{proposition}
\begin{proof}
We calculate
\begin{align}\label{Eisleveleq15 char}
    B_{N,\chi}&=\bigg(\prod_{p\mid N}\frac1{1-p}\bigg)\bigg(\bigotimes_{p\mid u}\delta_p\bigg)\otimes\bigg(\bigotimes_{p\mid N}(\alpha_p-\chi_p(p)\beta_p)\bigg)\otimes\bigg(\bigotimes_{p\nmid uN}\alpha_p\bigg)\nonumber\\
    &=\frac{\mu(N)}{\varphi(N)}\sum_{M\mid N}\bigg(\prod_{p\mid M}\chi_p(p)\bigg)\mu(M)\tilde{B}_{M,\chi}\nonumber\\
    &\stackrel{\eqref{Diri2}}{=}\frac{\mu(N)}{\varphi(N)}\sum_{M\mid N}\xi(M)\mu(M)\tilde{B}_{M,\chi}\nonumber\\
    &=\frac{1}{\varphi(N)}\sum_{M\mid N}\xi(N/M)\mu(M)\tilde{B}_{N/M,\chi}.
\end{align}
Using M\"obius inversion, it follows that
\begin{equation}\label{Eisleveleq16 char}
    \tilde{B}_{N,\chi}=\xi(N)\sum_{M\mid N}\varphi(M)B_{M, \chi}.
\end{equation}
Note here that $\xi(N)=\xi(N)^{-1}$. We apply both sides of \eqref{Eisleveleq15 char} and \eqref{Eisleveleq16 char} to $f_{2,\chi}$, and get from \eqref{Eisleveleq12 char} that
\begin{equation}\label{Eisleveleq17-char}
 f_{2,N,\chi}=\frac{1}{\varphi(N)}\sum_{M\mid N}\xi(N/M)\mu(M)\tilde{f}_{2,N/M,\chi},\qquad \tilde f_{2,N,\chi}=\xi(N)\sum_{M\mid N}\varphi(M)f_{2,M, \chi}.
\end{equation}
Next we take the adelic Eisenstein series on both sides of these equations. Then the equality of the adelic Eisenstein series implies the equality of the classical Eisenstein series in \eqref{tildeENrelationpropeq1 char}.
\end{proof}
\subsection{Global representations generated by Eisenstein series with character}
In this section we continue to let $\xi$ be a primitive Dirichlet character with conductor $u>1$ and $\chi$ the corresponding character of $\Q^\times\backslash\A^\times$. Recall the classical Eisenstein series with character \eqref{Eisexampleseq6-char}. The following is a version of Corollary~\ref{Ektheoremcor} that takes the presence of the characters into account.
\begin{corollary}\label{Ektheoremcorchar}
 If $k\geq4$ is an even integer, then the $\H$-module $\pi$ generated by the automorphic form corresponding to $E_{k,\xi}$ is irreducible. We have $\pi\cong\bigotimes\pi_v$, with $\pi_\infty=\mathcal{D}_{k-1}^{\rm hol}$, the discrete series representation of lowest weight $k$, and $\pi_p=\chi_p|\cdot|_p^{(k-1)/2}\times\chi_p^{-1}|\cdot|_p^{(1-k)/2}$, an irreducible principal series representation, for all $p<\infty$.
\end{corollary}

For weight $2$ we have to be more careful because the summation \eqref{Fouriereq1} is no longer absolutely convergent. However the arguments of Section~\ref{Sec:globalrepEis}, in particular Lemma~\ref{HfinPhilemma} and Proposition~\ref{V12ppEinjectiveprop} ii), remain valid. The situation is actually easier, because $E_{2,\xi}$ is holomorphic, hence the one-dimensional space of constant functions in Theorem~\ref{E2reptheorem} is no longer present. The upshot is that the automorphic form $\Phi_{2,\chi}=E(\cdot,f_{2,\chi})$ corresponding to $E_{2,\xi}$ generates the same global representation as the function $f_{2,\chi}$. Thus we obtain the following results, where we recall that $\mathcal{D}_p=\St_{\GL(2,\Q_p)}$.

\begin{theorem}\label{E2reptheorem char}
 Let $\Phi_{2,\chi}=E(\cdot,f_{2,\chi})$ be the automorphic form corresponding to $E_{2,\xi}$. Then the global representation $\H\Phi_{2,\chi}$ generated by $\Phi_{2,\chi}$ is
 \begin{equation}\label{E2reptheoremeq1 char}
  \H\Phi_{2,\chi}\cong\mathcal{D}_1^{\rm hol}\,\otimes\,\bigotimes_{\substack{p\mid u\\\chi_p^2\neq 1}}V_{1/2, \chi_p}\,\otimes\,\bigotimes_{\substack{p\mid u\\\chi_p^2=1}}\chi_p\mathcal{D}_p\,\otimes\,\bigotimes_{p\nmid u}V_{1/2, \chi_p}
 \end{equation}
 as $\H$-modules.
\end{theorem}
\begin{theorem}\label{E2Nreptheorem char}
 For a square-free positive integer $N$ with $(u, N)=1$ and such that $\chi_p^2=1$ for $p\mid N$, let $\Phi_{2,N, \chi}=E(\cdot,f_{2,N, \chi})$ be the automorphic form corresponding to $E_{2,N, \xi}$. Then
 \begin{equation}\label{E2Nreptheoremeq1 char}
  \H\Phi_{2,N, \chi}\cong\mathcal{D}_1^{\rm hol}\,\otimes\,\bigotimes_{\substack{p\mid u\\\chi_p^2\neq 1}}V_{1/2, \chi_p}\,\otimes\,\bigotimes_{\substack{p\mid u\\\chi_p^2=1}}\chi_p\mathcal{D}_p\,\otimes\,\bigotimes_{p\mid N}\chi_p\mathcal{D}_p\,\otimes\,\bigotimes_{p\nmid uN}V_{1/2, \chi_p}.
 \end{equation}
\end{theorem}
In terms of Dirichlet characters, the condition $\chi_p^2=1$ for $p\mid u$ can be detected as follows. Let $\eta$ be the primitive Dirichlet character corresponding to $\xi^2$, and let $u_1\mid u$ be the conductor of $\eta$. Then $\chi^2$ is the character of $\Q^\times\backslash\A^\times$ corresponding to $\eta$. Hence $\chi_p^2=1$ if and only if $p\nmid u_1$ and $\eta(p)=1$.
\bibliographystyle{plain}
\bibliography{E2_auto.bib}

\vspace{5ex}
\noindent Department of Mathematics, Lafayette College, Easton, Pennsylvania, USA.

\noindent E-mail address: {\tt royma@lafayette.edu}

\vspace{2ex}
\noindent Department of Mathematics, University of North Texas, Denton, TX 76203-5017, USA.

\noindent E-mail address: {\tt ralf.schmidt@unt.edu}

\vspace{2ex}
\noindent School of Mathematical Sciences, Xiamen University, Xiamen, Fujian 361005, PR China.

\noindent E-mail address: {\tt yishaoyun926@xmu.edu.cn}

\end{document}